\documentclass[11pt,a4paper]{article}

\usepackage[margin=1in]{geometry}
\usepackage[T1]{fontenc}
\usepackage[utf8]{inputenc}
\usepackage{lmodern}
\usepackage{microtype}
\usepackage{xcolor}
\definecolor{revisionpurple}{RGB}{128,0,128}

\usepackage{amsmath,amssymb,amsfonts,amsthm,mathtools,bm}
\usepackage{booktabs}
\usepackage{enumitem}
\usepackage[numbers,sort&compress]{natbib}
\usepackage[colorlinks=true,linkcolor=blue,citecolor=blue,urlcolor=blue]{hyperref}
\usepackage[nameinlink,capitalize,noabbrev]{cleveref}

\setlist[itemize]{leftmargin=2em}
\setlist[enumerate]{leftmargin=2em}

\newtheorem{theorem}{Theorem}[section]
\newtheorem{proposition}[theorem]{Proposition}
\newtheorem{lemma}[theorem]{Lemma}

\newtheorem{assumption}[theorem]{Assumption}

\theoremstyle{definition}
\newtheorem{definition}[theorem]{Definition}
\newtheorem{remark}[theorem]{Remark}

\newcommand{\R}{\mathbb{R}}

\newcommand{\E}{\mathbb{E}}
\newcommand{\Prob}{\mathbb{P}}
\newcommand{\calE}{\mathcal{E}}
\newcommand{\calP}{\mathcal{P}}
\newcommand{\calL}{\mathcal{L}}
\newcommand{\calM}{\mathcal{M}}
\newcommand{\dd}{\,\mathrm{d}}

\newcommand{\clock}{\mathsf{s}}

\newcommand{\softmax}{\mathrm{softmax}}
\newcommand{\diag}{\mathrm{diag}}
\newcommand{\supp}{\mathrm{supp}}
\newcommand{\argmin}{\mathrm{argmin}}
\newcommand{\Lip}{\mathrm{Lip}}
\newcommand{\Var}{\mathrm{Var}}

\newcommand{\inner}[2]{\left\langle #1,#2 \right\rangle}
\newcommand{\norm}[1]{\left\lVert #1 \right\rVert}
\newcommand{\abs}[1]{\left\lvert #1 \right\rvert}

\title{\bfseries
A Mean-Field Analysis of Multi-Head Self-Attention\\
under Cross-Entropy Training
}

\author{
Cheng Huan\thanks{Department of Statistics and Data Science,
The Chinese University of Hong Kong. Email: chenghuan@cuhk.edu.hk} and Hongwei Yuan\thanks{Corresponding author. Department of Mathematics, University of Macau. Email: hwyuan@um.edu.mo}}
\date{}

\begin{document}

\maketitle

\begin{abstract}
This paper develops a mean-field theory for a simplified single-layer causal multi-head self-attention model trained by cross-entropy minimization. Each attention head is treated as a particle in parameter space, and the empirical law of the heads is used as the large-head state variable. In the infinite-head limit, the averaged attention logits define a risk functional on probability measures, whose first variation generates a nonlinear Wasserstein gradient-flow equation. Unlike classical mean-field analyses of shallow networks that often focus on square-loss regression, the present model contains the softmax residual from the cross-entropy objective and the query-key-value structure of masked self-attention. We prove a static finite-head approximation bound for the optimal risk, characterize global minimizers through a variational support condition, and establish a quantitative finite-time propagation-of-chaos estimate comparing finite-head stochastic gradient descent with the limiting PDE. We then study the long-time behavior of the PDE: energy dissipation, convergence to the stationary set under compactness, convergence to a single stationary measure under topological or Kurdyka--{\L}ojasiewicz assumptions, and explicit convergence rates under gradient-domination conditions. Finally, we prove local exponential stability under a Wasserstein strong-monotonicity condition and give verifiable stability and instability criteria for Dirac stationary measures. The results provide a rigorous baseline mean-field framework for attention-head training and clarify the additional compactness, landscape, and curvature assumptions needed to pass from stationarity to convergence and stability.
\end{abstract}


\section{Introduction}

Attention mechanisms were introduced as trainable alignment maps for sequence-to-sequence models by \citet{BahdanauChoBengio2015}. The Transformer architecture of \citet{Vaswani2017} elevated self-attention from an auxiliary alignment module to the main mechanism for exchanging information between tokens. Since then, self-attention has become a central component of modern representation learning and language modeling, including BERT-type bidirectional pretraining \citep{DevlinChangLeeToutanova2019}, large autoregressive language models \citep{BrownMannRyder2020}, and empirical scaling-law studies \citep{KaplanMcCandlishHenighan2020}. These models are commonly trained with cross-entropy objectives for masked-token or next-token prediction. Despite this empirical success, the mathematical description of how attention heads move during training remains far less developed than the corresponding theory for shallow neural networks.

Mean-field theory offers a natural way to study overparameterized models by interpreting neurons, units, or attention heads as particles in parameter space. The empirical distribution of these particles may converge, under suitable scaling, to a deterministic measure-valued evolution. This viewpoint has been developed extensively for two-layer and shallow networks; see, for example, \citet{MeiMontanariNguyen2018}, \citet{SirignanoSpiliopoulos2020}, \citet{RotskoffVandenEijnden2022}, and the optimal-transport formulation of \citet{ChizatBach2018}. The present work adapts this philosophy to a masked multi-head self-attention model trained with a cross-entropy loss. This is different from the square-loss regression setting studied in much of the classical shallow-network mean-field literature, including \cite{MeiMontanariNguyen2018, SirignanoSpiliopoulos2020, RotskoffVandenEijnden2022}: here the loss is the token-level cross-entropy loss, and its gradient contributes the softmax residual \(\softmax(z_l^\rho)-Y_l\) rather than a linear residual from a squared error. The adaptation is not merely notational: the drift contains the softmax residual \(\softmax(z_l^\rho)-Y_l\), the feature map \(\psi_l(X;\omega)\) is generated by query-key-value attention, and all heads interact through the shared vocabulary logits. Consequently, the relevant limiting equation is a nonlinear transport equation on the space of probability measures over attention-head parameters.

The mean-field regime considered here is different from the neural tangent kernel (NTK) or lazy-training regimes of \citet{JacotGabrielHongler2018} and \citet{ChizatOyallonBach2019}. In the NTK description, the model is effectively linearized around initialization and training is governed by an almost fixed tangent kernel. In contrast, our mean-field scaling keeps the nonlinear movement of the empirical distribution of attention heads. This distinction is important for attention models because contextual relations are represented by the redistribution of heads in the query-key-value parameter space. The large-system limit is taken in the number of heads, while the internal dimension of a single head is fixed. Thus the normalization that keeps logits bounded is the \(1/N\) average over heads, not an architectural constraint such as \(d_k=d_{\mathrm{model}}/N\). The latter relation is useful in practical Transformer implementations to keep the total computational budget fixed, but it is not needed for the mathematical limit studied here.

We study a simplified single-layer masked self-attention architecture without feed-forward blocks, residual connections, or layer normalization. For a head parameter
\begin{align*}
\omega=(W_Q,W_K,W_V,W_O),
\end{align*}
the associated logit feature is
\begin{align*}
\psi_l(X;\omega)=\bigl[A_\omega(X)XW_VW_OW_{\mathrm{out}}\bigr]_{l,:},
\end{align*}
where \(A_\omega(X)\) is the row-wise masked attention matrix. For a probability measure \(\rho\) over head parameters, the mean-field logits are obtained by averaging this feature:
\begin{align*}
z_l^\rho(X)=\int \psi_l(X;\omega)\rho(\dd\omega).
\end{align*}
The cross-entropy risk \(\calE(\rho)\) is then a functional on \(\calP(K)\), and its first variation \(\Phi(\omega;\rho)\) generates the limiting Wasserstein gradient-flow equation
\begin{equation}
\label{eq:intro-pde}
\partial_t\rho_t
=
2\zeta(t)\nabla_\omega\cdot
\Bigl(
\rho_t\nabla_\omega\Phi(\omega;\rho_t)
\Bigr).
\end{equation}
This equation is the central object of the paper. It describes the deterministic evolution of the distribution of attention heads in the infinite-head limit and provides the reference dynamics for finite-head stochastic gradient descent.

The paper contains four groups of results. First, we prove a static mean-field approximation theorem (see \Cref{thm:static-approx}) showing that the optimal finite-head risk approaches the mean-field optimum at the Monte Carlo rate \(O(N^{-1/2})\). We also characterize global minimizers by a support condition (see \Cref{thm:minimizer-support}): a minimizer must put mass only on global minimizers of the first variation \(\Phi(\cdot;\rho_\ast)\). This condition is stronger than the stationarity condition and is essential for distinguishing convergence to stationary measures from convergence to globally optimal predictors.

Second, we derive the mean-field training dynamics and prove a quantitative propagation-of-chaos estimate comparing finite-head SGD with the limiting PDE over a fixed time horizon (see \Cref{thm:sgd-to-pde}). The error separates the Euler discretization contribution, the stochastic-gradient martingale contribution, and the empirical-measure fluctuation of the independent nonlinear characteristics. This finite-time theorem is the dynamical counterpart of the static approximation result and explains precisely how the number of heads, the learning-rate step size, and the probability parameter enter the approximation.

Third, we analyze the long-time behavior of the limiting PDE. We prove energy dissipation (see \Cref{thm:energy-dissipation}), characterize stationary solutions (see \Cref{thm:stationary}), establish convergence to the stationary set under \(W_1\)-precompactness (see \Cref{thm:omega-limit}), and give a checkable compact-support criterion for this precompactness (see Proposition \ref{prop:w1-precompactness}). We then show that convergence to a single stationary measure follows when the relevant stationary set is totally disconnected (see statement 5 in \Cref{thm:omega-limit}), and we derive explicit rates under a Kurdyka-{\L}ojasiewicz (KL) inequality (see \Cref{thm:kl-rate}).

Fourth, we study stability of stationary measures. A local Wasserstein strong monotonicity condition yields exponential stability in \(W_2\) (see \Cref{thm:wasserstein-stability}). For Dirac stationary measures, we give a more concrete criterion: positivity of the local Hessian \(H_0=\nabla_\omega^2\Phi(\omega_\ast;\delta_{\omega_\ast})\) implies local exponential stability of \(\delta_{\omega_\ast}\), while a negative direction of the translation Hessian \(H_{\mathrm{tr}}\) implies instability (see \Cref{thm:dirac-stability} and \Cref{thm:instability}). These results make explicit which curvature conditions are sufficient for stable attention-head concentration and which directions can destabilize a stationary head.

The analysis is intentionally restricted to a simplified architecture. The feed-forward block, residual connection, layer normalization, and trainable output projection are omitted so that the interaction among attention heads can be isolated. This restriction should be viewed as a model problem rather than a full theory of large language models. Nevertheless, the simplified setting retains several core mechanisms of attention training: masked token-to-token interaction, averaging over heads, cross-entropy residuals, nonlinear mean-field drift, and the gap between stationarity and optimality. The results therefore provide a rigorous baseline for future mean-field analyses of more complete Transformer architectures.

\section{Notation and model}

\subsection{Basic notation}

Let \(d_\omega\) denote the dimension of the parameter vector of a single attention head. The parameter space is a subset of \(\R^{d_\omega}\). We write \(\calP(\R^{d_\omega})\) for the space of Borel probability measures on \(\R^{d_\omega}\). For \(p\geq 1\), \(W_p\) denotes the \(p\)-Wasserstein distance.

For a finite collection of parameters \(\omega_1,\ldots,\omega_N\), the empirical measure is
\begin{equation}
\label{eq:emp-measure}
\widehat\rho^N
=
\frac{1}{N}\sum_{i=1}^N\delta_{\omega_i}.
\end{equation}
For a measurable function \(f\), we write
\begin{align*}
\langle f,\rho\rangle=\int f(\omega)\rho(\dd\omega).
\end{align*}

For \(z\in\R^d\), define
\begin{equation}
\label{eq:softmax}
\softmax_j(z)
=
\frac{\exp(z_j)}{\sum_{m=1}^{d}\exp(z_m)},
\qquad
j=1,\ldots,d.
\end{equation}
If \(p=\softmax(z)\), the Jacobian of the softmax map is
\begin{equation}
\label{eq:softmax-jacobian}
\nabla\softmax(z)
=
\diag(p)-pp^\top.
\end{equation}

\begin{lemma}[Lipschitz property of softmax]
\label{lem:softmax-lip}
For every \(z\in\R^{d}\),
\begin{align*}
0\preceq \nabla\softmax(z)\preceq \frac{1}{2}I.
\end{align*}
Consequently,
\begin{align*}
\norm{\softmax(z)-\softmax(z')}_2
\leq
\frac{1}{2}\norm{z-z'}_2
\end{align*}
for all \(z,z'\in\R^{d}\).
\end{lemma}

\begin{proof}
Let \(p=\softmax(z)\). For any \(u\in\R^{d}\),
\begin{align*}
u^\top\bigl(\diag(p)-pp^\top\bigr)u
=
\sum_{j=1}^{d}p_j u_j^2
-
\left(\sum_{j=1}^{d}p_j u_j\right)^2
=
\Var_p(u_J),
\end{align*}
where \(J\) is a random index with law \(p\). Hence the matrix is positive semidefinite. Moreover,
\begin{align*}
\Var_p(u_J)
\leq
\frac{1}{4}\bigl(\max_j u_j-\min_j u_j\bigr)^2
\leq
\frac{1}{2}\norm{u}_2^2,
\end{align*}
because \(\max_j u_j-\min_j u_j\leq \sqrt{2}\norm{u}_2\). This proves the spectral bound. The Lipschitz estimate follows from the mean-value theorem.
\end{proof}

\subsection{Single-head attention feature map}

Let
\begin{align*}
X=(x_1,\ldots,x_L)^\top\in\R^{L\times d_x}
\end{align*}
be an input token sequence. We consider a simplified masked self-attention layer without feed-forward block, residual connection, or layer normalization.

The parameter of a single attention head is
\begin{align*}
\omega=(W_Q,W_K,W_V,W_O)\in \mathbb{R}^{d_x \times d_k} \times \mathbb{R}^{d_x \times d_k} \times \mathbb{R}^{d_x \times d_k} \times \mathbb{R}^{d_k \times d_x}.
\end{align*}
For notational simplicity, all matrices are absorbed into a vector \(\omega\in\R^{d_\omega}\). Given \(\omega\), define
\begin{align*}
Q=XW_Q,
\qquad
K=XW_K,
\qquad
V=XW_V.
\end{align*}
Let \(\calM\in(\{0,-\infty\})^{L\times L}\) be the causal mask defined by
\begin{equation}
\label{eq:causal-mask}
\calM_{rs}
=
\begin{cases}
0, & 1\leq s\leq r\leq L,\\
-\infty, & 1\leq r<s\leq L.
\end{cases}
\end{equation}
Thus, at position \(r\), tokens \(s\leq r\) are visible and future tokens \(s>r\) receive attention weight zero after the row-wise softmax. The row-wise attention matrix is
\begin{equation}
\label{eq:attention-matrix}
A_\omega(X)
=
\softmax_{\mathrm{row}}
\left(
\frac{QK^\top}{\sqrt{d_k}}+\calM
\right).
\end{equation}
The single-head output after the output projection and the fixed vocabulary projection is
\begin{equation}
\label{eq:single-head-feature}
\Psi(X;\omega)
=
A_\omega(X)XW_VW_OW_{\mathrm{out}}
\in\R^{L\times d_v}.
\end{equation}
We denote the \(l\)-th row by
\begin{align*}
\psi_l(X;\omega)\in\R^{d_v},
\end{align*}
and its \(j\)-th coordinate by \(\psi_{l,j}(X;\omega)\).

In practical language models, these matrices have the following roles. The query matrix \(W_Q\) maps each token representation to the vector used to search for relevant previous tokens. The key matrix \(W_K\) maps tokens to vectors against which queries are compared, so the scores \(QK^\top/\sqrt{d_k}\) determine which context positions are attended to. The value matrix \(W_V\) maps tokens to the information that is aggregated after the attention weights are chosen. The output projection \(W_O\) maps the aggregated value of one head back to the model dimension and mixes the coordinates produced inside the head. Finally, \(W_{\mathrm{out}}\in\R^{d_x\times d_v}\) maps the model representation to vocabulary logits. In full Transformer implementations, \(W_{\mathrm{out}}\) is usually trained jointly with the attention parameters, but here it is fixed so that the only interacting particles are the attention heads.

% \begin{remark}
% The fixed matrix \(W_{\mathrm{out}}\) is introduced to keep the analysis focused on the mean-field behavior of the attention heads. 
% \end{remark}

\subsection{Mean-field logits and cross-entropy risk}

Let \(Y_l(X)\in\R^{d_v}\) be the one-hot next-token label at position \(l\). For \(\rho\in\calP(\R^{d_\omega})\), define the mean-field logits
\begin{equation}
\label{eq:mean-field-logits}
z_l^\rho(X)
=
\int_{\R^{d_\omega}}\psi_l(X;\omega)\rho(\dd\omega)
\in\R^{d_v}.
\end{equation}

For \(y\in\R^{d_v}\) one-hot and \(z\in\R^{d_v}\), define the cross-entropy loss
\begin{equation}
\label{eq:ce-loss}
\ell(y,z)
=
-\sum_{j=1}^{d_v}y_j\log\softmax_j(z)
=
\log\left(\sum_{j=1}^{d_v}e^{z_j}\right)-y^\top z.
\end{equation}
Then
\begin{align*}
\nabla_z\ell(y,z)=\softmax(z)-y.
\end{align*}
The mean-field risk is
\begin{equation}
\label{eq:mean-field-risk}
\calE(\rho)
=
\E_X
\left[
\frac{1}{L}\sum_{l=1}^L
\ell\bigl(Y_l(X),z_l^\rho(X)\bigr)
\right].
\end{equation}
For a finite \(N\)-head model, define
\begin{align*}
\calE_N(\omega_1,\ldots,\omega_N)
=
\calE(\widehat\rho^N).
\end{align*}

\subsection{First variation}

The first variation of \(\calE\) is the function
\begin{equation}
\label{eq:first-variation}
\Phi(\omega;\rho)
=
\E_X
\left[
\frac{1}{L}
\sum_{l=1}^L
\left\langle
\softmax(z_l^\rho(X))-Y_l(X),
\psi_l(X;\omega)
\right\rangle
\right].
\end{equation}
Indeed, if \(\rho_\varepsilon=(1-\varepsilon)\rho+\varepsilon\nu\), then
\begin{equation}
\label{eq:first-var-identity}
\left.
\frac{\dd}{\dd\varepsilon}
\calE(\rho_\varepsilon)
\right|_{\varepsilon=0}
=
\int \Phi(\omega;\rho)(\nu-\rho)(\dd\omega).
\end{equation}
Moreover,
\begin{equation}
\label{eq:grad-phi}
\nabla_\omega\Phi(\omega;\rho)
=
\E_X
\left[
\frac{1}{L}
\sum_{l=1}^L
\nabla_\omega\psi_l(X;\omega)^\top
\bigl(
\softmax(z_l^\rho(X))-Y_l(X)
\bigr)
\right],
\end{equation}
where \(\nabla_\omega\psi_l(X;\omega)\in\R^{d_v\times d_\omega}\).

\section{Assumptions}

The following assumptions are used throughout the paper.

\begin{assumption}[Learning-rate profile]
\label{ass:learning-rate}
The function \(\zeta:[0,\infty)\to[0,\infty)\) is bounded and Lipschitz. Moreover,
\begin{align*}
\int_0^\infty \zeta(t)\dd t=\infty.
\end{align*}
When exponential convergence in physical time is stated, we additionally assume that there exists \(\zeta_->0\) such that
\begin{align*}
\zeta(t)\geq \zeta_-
\qquad
\text{for all }t\geq 0.
\end{align*}
\end{assumption}

\begin{assumption}[Smoothness on invariant compact sets]
\label{ass:smoothness}
There exist two compact sets \(K\subset\R^{d_\omega}\), \(K_x\subset\R^{d_x}\) such that all measures considered below are supported in \(K\) and all input tokens considered below are in \(K_x\). On \(K\), the maps
\begin{align*}
\omega\mapsto \psi_l(X;\omega)
\end{align*}
are \(C^3\) for all \(l\in\{1,\ldots,L\}\) and almost every \(X\in K_x^{L}\), and there exists \(B_\psi<\infty\) %(depending on $d_v$) 
such that
\begin{align*}
\sup_{X\in K_x^{L},l\in\{1,\ldots,L\},\omega\in K}
\left(
\norm{\psi_l(X;\omega)}
+
\norm{\nabla_\omega\psi_l(X;\omega)}
+
\norm{\nabla_\omega^2\psi_l(X;\omega)}
+
\norm{\nabla_\omega^3\psi_l(X;\omega)}
\right)
\leq B_\psi.
\end{align*}
\end{assumption}

\begin{remark}
For finite-time propagation-of-chaos estimates, it is enough to assume that the particle trajectories remain in a compact set on the relevant time interval. For the long-time convergence results, compactness can be guaranteed either by an invariant compact parameter domain, by an a priori confinement estimate, or locally by the stability theorem in \Cref{thm:dirac-stability}.
\end{remark}

\begin{proposition}[Attention features satisfy the smoothness assumption locally]
\label{prop:attention-smooth}

Assume that \(\norm{X}\leq M_X\), \(\norm{\omega}\leq M_\omega\), and \(W_{\mathrm{out}}\) is fixed. Then the attention feature map \(\omega\mapsto\psi_l(X;\omega)\) defined in \eqref{eq:single-head-feature} is smooth. Moreover, for every integer \(m\geq0\) there exists a constant
\begin{align*}
C_m=C_m(M_X,M_\omega,\norm{W_{\mathrm{out}}},d_k,L)
\end{align*}
such that
\begin{align*}
\sup_{\norm{X}\leq M_X,l\in\{1,\ldots,L\},\norm{\omega}\leq M_\omega}
\norm{\nabla_\omega^m\psi_l(X;\omega)}
\leq C_m.
\end{align*}
In particular, Assumption \ref{ass:smoothness} holds for \(K=B(0,M_\omega)\) and \(K_x=B(0,M_X)\).
\end{proposition}
\begin{remark}
In $C_m(M_X,M_\omega,\norm{W_{\mathrm{out}}},d_k,L)$, \(L\) appears because each row of the row-wise softmax is a function of \(L\) attention scores and the derivatives of \(A_\omega(X)\) involve sums over the \(L\) token positions. There is no separate dependence on \(d_x\) or \(d_v\) because \(M_X\), \(M_\omega\), and \(\norm{W_{\mathrm{out}}}\) are dimension-aware norm bounds: the dependence on the input dimension is absorbed by the bounds on \(X\) and the parameter matrices, while the dependence on the vocabulary dimension is absorbed by the operator norm of the fixed linear map \(W_{\mathrm{out}}\). 
\end{remark}

\begin{proof}
The map \(\omega\mapsto QK^\top/\sqrt{d_k}\) is polynomial in \(\omega\), and the row-wise softmax is \(C^\infty\) with bounded derivatives on bounded sets. The products with \(XW_V\), \(W_O\), and \(W_{\mathrm{out}}\) are polynomial operations. Therefore the claim follows from repeated applications of the chain rule and product rule. The bounds depend only on the displayed quantities. %The first derivative bound also follows directly from Lemma \ref{lem:softmax-lip}.
\end{proof}

\section{Static mean-field approximation}

We first compare the finite-head optimization problem with the mean-field variational problem.

\begin{theorem}[Approximation of the optimal risk]
\label{thm:static-approx}
Under Assumption \ref{ass:smoothness}, there exists a constant \(C<\infty\), depending only on \(B_\psi\) and \(L\), such that
\begin{equation}
\label{eq:static-approx}
0
\leq
\inf_{\omega_1,\ldots,\omega_N\in K}\calE_N(\omega_1,\ldots,\omega_N)
-
\inf_{\rho\in\calP(K)}\calE(\rho)
\leq
\frac{C}{\sqrt{N}}.
\end{equation}
%The dependence on \(L\) comes from controlling the loss uniformly over the \(L\) prediction positions, and the dependence on \(d_v\) comes from estimating deviations of \(d_v\)-dimensional logit vectors in Euclidean norm. If \(B_\psi\) is chosen as a dimension-free bound for the full vector \(\psi_l\), then the \(d_v\)-dependence is absorbed into \(B_\psi\) and hence into \(C\).
\end{theorem}

\begin{remark}[Comparison with classical mean-field approximation]
This theorem is the static analogue of the approximation step used in mean-field analyses of shallow networks such as \citet{MeiMontanariNguyen2018}, \citet{SirignanoSpiliopoulos2020}, and \citet{RotskoffVandenEijnden2022}. In those works the averaged particles are usually neurons or hidden units, whereas here each particle is an attention head and the averaged object is a vocabulary-logit feature \(\psi_l(X;\omega)\). The present estimate is deliberately stated at the Monte Carlo scale \(O(N^{-1/2})\), because the proof controls
$\E\norm{N^{-1}\sum_{i=1}^N U_i}$
by the standard central-limit-size fluctuation of an empirical average and then uses only the Lipschitz continuity of the cross-entropy loss with respect to the logits. By contrast, the sharper \(O(N^{-1})\)-type bounds that can appear in some two-layer mean-field analyses, including estimates in \citet{MeiMontanariNguyen2018}, exploit additional quadratic or second-order structure, for example a square loss, cancellations after expanding a squared residual, or a more restrictive smoothness/regularity argument for the risk itself. Such cancellations are not used here. The cross-entropy objective contains a logarithmic softmax term and a one-hot linear term; it is globally Lipschitz on bounded logit sets but it is not treated as a uniformly quadratic loss in this argument. Thus the rate reflects the generic empirical approximation of the attention-head logit feature rather than a special second-order approximation of a squared-loss model.
\end{remark}

\begin{proof}
The lower bound is immediate because every empirical measure is an element of \(\calP(K)\).

For the upper bound, fix \(\eta>0\) and choose \(\rho_\eta\in\calP(K)\) such that
\begin{align*}
\calE(\rho_\eta)
\leq
\inf_{\rho\in\calP(K)}\calE(\rho)+\eta.
\end{align*}
Let \(\omega_1,\ldots,\omega_N\) be independent samples from \(\rho_\eta\), and set \(\widehat\rho^N=N^{-1}\sum_i\delta_{\omega_i}\). Fix \(X\) and \(l\). Since
\begin{align*}
z_l^{\widehat\rho^N}(X)-z_l^{\rho_\eta}(X)
=
\frac{1}{N}\sum_{i=1}^N
\left(
\psi_l(X;\omega_i)
-
\E_{\omega\sim\rho_\eta}\psi_l(X;\omega)
\right),
\end{align*}
set
\begin{align*}
U_i(X,l)
=
\psi_l(X;\omega_i)
-
\E_{\omega\sim\rho_\eta}\psi_l(X;\omega).
\end{align*}
Then \(U_i(X,l)\) are independent, mean-zero \(\R^{d_v}\)-valued random variables. By Jensen's inequality and independence,
\begin{align*}
\E_{\omega_1,\ldots,\omega_N}
\norm{
z_l^{\widehat\rho^N}(X)-z_l^{\rho_\eta}(X)
}
&\leq
\left(
\E_{\omega_1,\ldots,\omega_N}
\norm{
\frac{1}{N}\sum_{i=1}^N U_i(X,l)
}^2
\right)^{1/2}\\
&=
\left(
\frac{1}{N^2}
\sum_{i=1}^N
\E\norm{U_i(X,l)}^2
\right)^{1/2},
\end{align*}
because the cross terms vanish. Assumption \ref{ass:smoothness} gives
\(\norm{\psi_l(X;\omega)}\leq B_\psi\), hence
\begin{align*}
\norm{U_i(X,l)}
\leq
\norm{\psi_l(X;\omega_i)}
+
\E_{\omega\sim\rho_\eta}\norm{\psi_l(X;\omega)}
\leq
2B_\psi.
\end{align*}
Therefore
\begin{align}
\label{eq:logit-mc-bound}
\E_{\omega_1,\ldots,\omega_N}
\norm{
z_l^{\widehat\rho^N}(X)-z_l^{\rho_\eta}(X)
}
\leq
\frac{2B_\psi}{\sqrt{N}}.
\end{align}
The bound is uniform in \(X\) and \(l\). Moreover,
\begin{align*}
\norm{\nabla_z\ell(y,z)}_2
=
\norm{\softmax(z)-y}_2
\leq 2,
\end{align*}
so \(z\mapsto \ell(y,z)\) is \(2\)-Lipschitz. Therefore
\begin{align*}
\E_{\omega_1,\ldots,\omega_N}
\left[
\calE(\widehat\rho^N)-\calE(\rho_\eta)
\right]
\leq
\frac{C}{\sqrt{N}}.
\end{align*}
Hence there exists one realization of \(\omega_1,\ldots,\omega_N\) for which
\begin{align*}
\calE_N(\omega_1,\ldots,\omega_N)
\leq
\calE(\rho_\eta)+\frac{C}{\sqrt{N}}.
\end{align*}
Letting \(\eta\downarrow0\) proves the result.
\end{proof}

\begin{theorem}[First-order characterization of global minimizers]
\label{thm:minimizer-support}
Let \(\rho_\ast\in\calP(K)\). Then \(\rho_\ast\) is a global minimizer of \(\calE\) over \(\calP(K)\) if and only if
\begin{equation}
\label{eq:minimizer-support}
\supp(\rho_\ast)
\subset
\argmin_{\omega\in K}\Phi(\omega;\rho_\ast).
\end{equation}
Equivalently,
\begin{align*}
\Phi(\omega;\rho_\ast)
\geq
\int \Phi(\omega';\rho_\ast)\rho_\ast(\dd\omega')
\quad
\text{for all }\omega\in K,
\end{align*}
with equality \(\rho_\ast\)-almost surely.
\end{theorem}

% \begin{remark}[Relation to variational optimality in measure space]
% The support condition is the probability-measure analogue of the first-order optimality conditions used in optimal-transport formulations of overparameterized models, for example \citet{ChizatBach2018}. Its interpretation is different from the finite-dimensional critical-point condition \(\nabla_{\omega_i}\calE_N=0\): a mean-field minimizer must place mass only where the first variation \(\Phi(\cdot;\rho_\ast)\) is minimal. This also differs from the NTK or lazy-training picture of \citet{JacotGabrielHongler2018} and \citet{ChizatOyallonBach2019}, where the features are effectively frozen near initialization.
% \end{remark}

\begin{proof}
The functional \(\calE\) is convex in \(\rho\), because \(z_l^\rho(X)\) is linear in \(\rho\) and the cross-entropy loss \(\ell(y,z)\) is convex in \(z\). Thus \(\rho_\ast\) is a global minimizer if and only if
\begin{align*}
\left.
\frac{\dd}{\dd\varepsilon}
\calE\bigl((1-\varepsilon)\rho_\ast+\varepsilon\nu\bigr)
\right|_{\varepsilon=0}
\geq 0
\end{align*}
for all \(\nu\in\calP(K)\). By \eqref{eq:first-var-identity}, this is equivalent to
\begin{align*}
\int\Phi(\omega;\rho_\ast)\nu(\dd\omega)
\geq
\int\Phi(\omega;\rho_\ast)\rho_\ast(\dd\omega)
\end{align*}
for all \(\nu\in\calP(K)\). Taking \(\nu=\delta_\omega\) gives
\begin{align*}
\Phi(\omega;\rho_\ast)
\geq
\int\Phi(\omega';\rho_\ast)\rho_\ast(\dd\omega')
\quad
\text{for all }\omega\in K.
\end{align*}
Since equality must hold after integrating with respect to \(\rho_\ast\), the support of \(\rho_\ast\) is contained in the set of minimizers of \(\Phi(\cdot;\rho_\ast)\). The converse follows immediately from the same variational inequality.
\end{proof}

\section{Mean-field training dynamics}

\subsection{Finite-head SGD}

For the finite-head model, let
\begin{align*}
\widehat\rho_k^N
=
\frac{1}{N}\sum_{i=1}^N\delta_{\omega_i^k}.
\end{align*}
Given an independent data sample \(X^{k+1}\), define the stochastic gradient
\begin{equation}
\label{eq:stoch-grad}
G_i(\omega^k;X^{k+1})
=
\frac{1}{L}\sum_{l=1}^L
\nabla_\omega\psi_l(X^{k+1};\omega_i^k)^\top
\left(
Y_l(X^{k+1})
-
\softmax(z_l^{\widehat\rho_k^N}(X^{k+1}))
\right).
\end{equation}
Thus
\begin{align}\label{eq:mean-zero_1}
\E\left[
G_i(\omega^k;X^{k+1})\mid \omega^k
\right]
=
-\nabla_\omega\Phi(\omega_i^k;\widehat\rho_k^N).
\end{align}
To match the normalization of \eqref{eq:intro-pde}, we write the SGD update as
\begin{equation}
\label{eq:sgd}
\omega_i^{k+1}
=
\omega_i^k
+
2\varepsilon\zeta(k\varepsilon)
G_i(\omega^k;X^{k+1}).
\end{equation}
The factor \(2\) is a convention and could be absorbed into \(\zeta\).

% A noisy version with weight decay is
% \begin{equation}
% \label{eq:noisy-sgd}
% \omega_i^{k+1}
% =
% \omega_i^k
% +
% 2\varepsilon\zeta(k\varepsilon)
% \left(
% G_i(\omega^k;X^{k+1})-\lambda\omega_i^k
% \right)
% +
% \sqrt{4\beta^{-1}\varepsilon\zeta(k\varepsilon)}\,\xi_i^k,
% \end{equation}
% where \(\xi_i^k\sim N(0,I_{d_\omega})\) are independent.

\subsection{Limiting PDE}

The deterministic mean-field PDE associated with \eqref{eq:sgd} is
\begin{equation}
\label{eq:pde}
\partial_t\rho_t
=
2\zeta(t)\nabla_\omega\cdot
\left(
\rho_t\nabla_\omega\Phi(\omega;\rho_t)
\right).
\end{equation}
Equivalently, if \(W_t\) solves
\begin{equation}
\label{eq:characteristic}
\dot W_t
=
-2\zeta(t)\nabla_\omega\Phi(W_t;\rho_t),
\qquad
\rho_t=\calL(W_t),
\end{equation}
then \(\rho_t\) solves \eqref{eq:pde} in the weak sense.

% The noisy SGD \eqref{eq:noisy-sgd} corresponds formally to the nonlinear Fokker--Planck equation
% \begin{equation}
% \label{eq:noisy-pde}
% \partial_t\rho_t
% =
% 2\zeta(t)\nabla_\omega\cdot
% \left(
% \rho_t\nabla_\omega\Phi_\lambda(\omega;\rho_t)
% \right)
% +
% 2\beta^{-1}\zeta(t)\Delta_\omega\rho_t,
% \qquad
% \Phi_\lambda(\omega;\rho)
% =
% \Phi(\omega;\rho)+\frac{\lambda}{2}\norm{\omega}^2.
% \end{equation}

\begin{definition}[Weak solution]
A narrowly continuous curve \(t\mapsto\rho_t\in\calP(K)\) is a weak solution of \eqref{eq:pde} if, for every \(f\in C^1(K)\),
\begin{equation}
\label{eq:weak-pde}
\frac{\dd}{\dd t}\int f(\omega)\rho_t(\dd\omega)
=
-2\zeta(t)
\int
\inner{\nabla f(\omega)}{\nabla_\omega\Phi(\omega;\rho_t)}
\rho_t(\dd\omega)
\end{equation}
in the sense of distributions in time.
\end{definition}

\subsection{Propagation-of-chaos estimate}

Define
\begin{equation}
\label{eq:error-term}
\begin{aligned}
\operatorname{Err}_{N,\varepsilon,T}(z)
:=&\ \varepsilon
+
\sqrt{\varepsilon}
\left[
\sqrt{d_\omega}
+
\sqrt{\log N}
+
z
\right]+
\frac{1}{\sqrt N}
\left[
\sqrt L
+
\sqrt{\log\bigl(\lfloor T/\varepsilon\rfloor+1\bigr)}
+
z
\right].
\end{aligned}
\end{equation}
The three terms correspond respectively to the Euler time-discretization error, the stochastic-gradient martingale error, and the empirical-measure error of the independent nonlinear characteristics. 
%This definition keeps the optimized square-root logarithmic factors coming from the choices of \(r\) in \eqref{Choose_r} and \(s\) in \eqref{Choose_s}.

\begin{theorem}[Mean-field approximation of SGD]
\label{thm:sgd-to-pde}
Assume Assumption \ref{ass:learning-rate} and \ref{ass:smoothness}. Let \(\omega_i^0,\,i\in\{1,\ldots,N\}\) be independent with common law \(\rho_0\), and let $\omega_i^k$ be defined by the SGD update \eqref{eq:sgd} and \(\rho_t\) be the solution of \eqref{eq:pde} with initial law \(\rho_0\). Then for every fixed \(T>0\), there exists \(C_T<\infty\), depending only on \(T\), the compact sets \(K\) and \(K_x\), the bounds in Assumptions \ref{ass:learning-rate} and \ref{ass:smoothness}, the sequence length \(L\) and the parameter dimension \(d_\omega\), but not on \(N\), \(\varepsilon\), \(z\), or the test function \(f\), such that, for every $f:\mathbb{R}^{d_\omega} \rightarrow \mathbb{R}$ with $\mathrm{Lip}(f) \leq 1$ and every $\varepsilon \leq 1$,
\begin{equation}
\label{eq:w1-sgd-bound}
\sup_{k \in [0 , \, T/\varepsilon] \cap \mathbb{N}} \bigg|\frac{1}{N} \sum_{i=1}^N f(\omega_i^k) - \int f(\omega) \rho_{k \varepsilon}(\mathrm{d}\omega) \bigg| 
\leq
C_T\operatorname{Err}_{N,\varepsilon,T}(z),
\end{equation}
and
\begin{equation}
\label{eq:risk-sgd-bound}
\sup_{k \in [0 , \, T/\varepsilon] \cap \mathbb{N}}
\abs{
\calE_N(\omega_1^k,\ldots,\omega_N^k)
-
\calE(\rho_{k\varepsilon})
}
\leq
C_T\operatorname{Err}_{N,\varepsilon,T}(z),
\end{equation}
with probability at least \(1-e^{-z^2}\). 
%The same estimates hold for the noisy SGD \eqref{eq:noisy-sgd}, with \eqref{eq:pde} replaced by \eqref{eq:noisy-pde}, provided the initial law is sub-Gaussian and \(\beta\geq 1\), \(\lambda\leq 1\).
\end{theorem}

\begin{remark}[Relation to propagation of chaos]
The estimate is a finite-time propagation-of-chaos statement in the spirit of \citet{Sznitman1991}. The distinctive point here is that the drift contains the cross-entropy residual \(\softmax(z_l^\rho)-Y_l\) and the masked attention feature map. Thus the theorem compares SGD for finitely many heads with a nonlinear continuity equation for the law of a representative attention head, rather than with a fixed kernel dynamics as in the NTK regime.
\end{remark}

\begin{proof}
Write
\begin{align*}
t_k=k\varepsilon,
\qquad
\zeta_k=\zeta(t_k),
\qquad
b(t,\omega,\rho)=-2\zeta(t)\nabla_\omega\Phi(\omega;\rho).
\end{align*}
Let \(\bar\omega_i(t)\) be the nonlinear characteristic
\begin{align*}
\dot{\bar\omega}_i(t)
=
b(t,\bar\omega_i(t),\rho_t),
\qquad
\bar\omega_i(0)=\omega_i^0.
\end{align*}
Since the initial variables are independent with common law \(\rho_0\) and the characteristic equation is driven only by the deterministic law \(\rho_t\), the variables \(\bar\omega_i(t)\) are independent with common law \(\rho_t\) for every \(t\geq0\).

We first record the Lipschitz estimates used below. By Assumption \ref{ass:smoothness}, Lemma \ref{lem:softmax-lip}, 
there is a constant \(C_K\) such that, for all \(\omega,\omega'\in K\) and \(\rho,\rho'\in\calP(K)\),
\begin{align}
\label{eq:phi-lip-proof}
\norm{
\nabla_\omega\Phi(\omega;\rho)-\nabla_\omega\Phi(\omega';\rho')
}
\leq
C_K\left(
\norm{\omega-\omega'}+\E_X
\left[\max_{l\in\{1,\ldots,L\}}\norm{\int_{\R^{d_\omega}}\psi_l(X;\omega)(\rho-\rho')(\dd\omega)
}\right]\right).
\end{align}
Indeed, the \(\omega\)-dependence is controlled by the uniform second-derivative bound on \(\psi_l\). For the measure argument, observe that
\begin{align*}
\norm{z_l^\rho(X)-z_l^{\rho'}(X)}
=
\norm{
\int \psi_l(X;\omega)(\rho-\rho')(\dd\omega)
}
\end{align*}
 and then the Lipschitz continuity of softmax gives \eqref{eq:phi-lip-proof}. The same estimates imply the local truncation bound
\begin{align}
\label{eq:char-local-error}
\norm{
\bar\omega_i(t_{k+1})-\bar\omega_i(t_k)
-\varepsilon b(t_k,\bar\omega_i(t_k),\rho_{t_k})
}
\leq
C_T\varepsilon^2,
\end{align}
where \(C_T\) depends on \(T\), the Lipschitz constant of \(\zeta\), and the constants in Assumption \ref{ass:smoothness}.

Define the error
\begin{align*}
e_i^k=\omega_i^k-\bar\omega_i(t_k)
\end{align*}
and the martingale increment
\begin{align*}
\Delta M_i^{k+1}
=
2\varepsilon\zeta_k
\left[
G_i(\omega^k;X^{k+1})
+
\nabla_\omega\Phi(\omega_i^k;\widehat\rho_k^N)
\right].
\end{align*}
Then \(\E[\Delta M_i^{k+1}\mid\omega^k]=0\) by \eqref{eq:mean-zero_1}. %Indeed, conditional on \(\omega^k\), the only remaining randomness is the fresh sample \(X^{k+1}\), which is independent of the past, and \eqref{eq:grad-phi} gives
% \begin{align*}
% \E\left[G_i(\omega^k;X^{k+1})\mid\omega^k\right]
% =
% -\nabla_\omega\Phi(\omega_i^k;\widehat\rho_k^N).
% \end{align*}
%Therefore the two terms inside the brackets in the definition of \(\Delta M_i^{k+1}\) have conditional expectation zero. 
Subtracting the characteristic increment from the SGD update gives
\begin{align*}
e_i^{k+1}
=&\ e_i^k
-2\varepsilon\zeta_k
\left[
\nabla_\omega\Phi(\omega_i^k;\widehat\rho_k^N)
-
\nabla_\omega\Phi(\bar\omega_i(t_k);\rho_{t_k})
\right]
+\Delta M_i^{k+1}
+r_i^{k+1},
\end{align*}
with \(\norm{r_i^{k+1}}\leq C_T\varepsilon^2\) by \eqref{eq:char-local-error}. Iterating this equality and using $e_i^0=0$ gives 
\begin{align*}
e_i^{k+1}
=&\ 
-2\varepsilon\sum_{m=0}^k\left(\zeta_m
\left[
\nabla_\omega\Phi(\omega_i^m;\widehat\rho_m^N)
-
\nabla_\omega\Phi(\bar\omega_i(t_m);\rho_{t_m})
\right]
+\Delta M_i^{m+1}
+r_i^{m+1}\right),
\end{align*}
Hence \eqref{eq:phi-lip-proof} yields
\begin{align}
\label{eq:error-recursion-one-step}
\norm{e_i^{k+1}}
\leq&
C_K\varepsilon\sum_{m=0}^k\left(\norm{e_i^m}+A_m\right)
+\norm{M_i^{k+1}}
+C_T\varepsilon^2(k+1).
\end{align}
where \(M_i^m=\sum_{q=0}^{m-1}\Delta M_i^{q+1}\) and
\begin{align*}
A_k
=
\E_X
\left[
\max_{l\in\{1,\ldots,L\}}
\norm{
\int_{\R^{d_\omega}} \psi_l(X;\omega)(\widehat\rho_k^N-\rho_{t_k})(\dd\omega)
}
\right].
\end{align*}
Using the discrete Gronwall inequality gives
\begin{align}
\label{eq:e-bound-before-gronwall}
\max_{i\leq N}\norm{e_i^k}
\leq
C_T\left[
\varepsilon
+
\max_{i\leq N}
\norm{M_i^k}
+
\varepsilon\sum_{m=0}^{k-1}
A_m
\right],
\end{align}

We next control the martingale term $\max_{i\leq N}
\norm{M_i^k}$. The softmax residual satisfies
\begin{align*}
\norm{Y_l(X)-\softmax(z_l^{\widehat\rho_k^N}(X))}\leq 2,
\end{align*}
and \(\nabla_\omega\psi_l\) is bounded on \(K\). Therefore
\begin{align*}
\norm{G_i(\omega^k;X^{k+1})}
+
\norm{\nabla_\omega\Phi(\omega_i^k;\widehat\rho_k^N)}
\leq C_K,
\end{align*}
so the martingale increments are bounded by \(C_K\varepsilon\). To derive the required martingale bound, fix \(i\) and a unit vector \(u\in\R^{d_\omega}\). The scalar process \(\inner{u}{M_i^k}\) is a martingale with increments
\begin{align*}
\abs{\inner{u}{\Delta M_i^{q+1}}}\leq C_K\varepsilon,
\end{align*}
and predictable quadratic variation bounded by
\begin{align*}
\sum_{q< T/\varepsilon}
\E\left[
\inner{u}{\Delta M_i^{q+1}}^2\mid \omega^q
\right]
\leq C_T\varepsilon.
\end{align*}
Therefore, Azuma--Hoeffding's inequality 
% says that if a scalar martingale has increments bounded by \(a_q\), then
% \begin{align*}
% \Prob\left(\sup_{m\leq n}\abs{S_m}>r\right)
% \leq
% 2\exp\left(-\frac{r^2}{2\sum_{q=1}^n a_q^2}\right),
% \end{align*}
% where the supremum version follows by applying the same inequality to the stopped martingale. Since here \(a_q\leq C_K\varepsilon\) and \(n\leq T/\varepsilon\), the denominator is bounded by \(C_T\varepsilon\). 
% Freedman's inequality gives the slightly sharper variance-sensitive form
% \begin{align*}
% \Prob\left(\sup_{m\leq n}S_m>r\right)
% \leq
% \exp\left(
% -\frac{r^2}{2(V_n+C_K\varepsilon r/3)}
% \right),
% \qquad
% V_n=\sum_{q=1}^n\E[(S_q-S_{q-1})^2\mid\mathcal F_{q-1}],
% \end{align*}
% and the bound \(V_n\leq C_T\varepsilon\) again yields the sub-Gaussian estimate, for \(0<r\leq 1\),
gives
\begin{align}
\label{eq:scalar-martingale-bound}
\Prob\left(
\sup_{0\leq k\leq T/\varepsilon}
\abs{\inner{u}{M_i^k}}>r
\right)
\leq
2\exp\left(-\frac{r^2}{C_T \varepsilon}\right).
\end{align}

Let \(\mathcal N\subset S^{d_\omega-1}\) be a \(1/2\)-net of the unit sphere, meaning that for every \(u\in S^{d_\omega-1}\) there exists \(v\in\mathcal N\) such that \(\norm{u-v}\leq1/2\). The standard volumetric covering argument gives \(\abs{\mathcal N}\leq (1+2/(1/2))^{d_\omega}=5^{d_\omega}\). For any \(w\in\R^{d_\omega}\), if \(u=w/\norm{w}\), choose \(v\in\mathcal N\) with \(\norm{u-v}\leq1/2\); then
\begin{align*}
\norm{w}
=
\inner{u}{w}
\leq
\abs{\inner{v}{w}}+\norm{u-v}\norm{w}
\leq
\max_{v\in\mathcal N}\abs{\inner{v}{w}}+\frac12\norm{w},
\end{align*}
and therefore \(\norm{w}\leq2\max_{v\in\mathcal N}\abs{\inner{v}{w}}\). Applying \eqref{eq:scalar-martingale-bound} to all \(v\in\mathcal N\) and all \(i\leq N\), we obtain
\begin{align*}
&\Prob\left(
\max_{i\leq N}\max_{v\in\mathcal N}
\sup_{0\leq k\leq T/\varepsilon}
\abs{\inner{v}{M_i^k}}>r
\right) \\
&\qquad\leq
2N\abs{\mathcal N}
\exp\left(-\frac{r^2}{C_T\varepsilon}\right)
\leq
2\exp\left(
\log N+d_\omega\log 5-\frac{r^2}{C_T\varepsilon}
\right).
\end{align*}
Set $2\exp\left(
\log N+d_\omega\log 5-\frac{r^2}{C_T\varepsilon}
\right)=\frac12 e^{-z^2}$, which means that we choose
\begin{align}\label{Choose_r}
r
=
\sqrt{C_T\varepsilon\left(z^2+\log N+d_\omega\log 5+\log 4\right)}.
\end{align}
Then 
\begin{align}\nonumber
    \max_{i\leq N}\sup_{0\leq k\leq T/\varepsilon}\norm{M_i^k}&\leq 2\max_{i\leq N}\max_{v\in\mathcal N}
\sup_{0\leq k\leq T/\varepsilon}
\abs{\inner{v}{M_i^k}}\\\nonumber
&<2r=2\sqrt{C_T\varepsilon\left(z^2+\log N+d_\omega\log 5+\log 4\right)}\\\label{eq:martingale-concentration}
&\leq C_T\sqrt{\varepsilon}\left[\sqrt{d_\omega}+\sqrt{\log N}+z\right]
\end{align} 
with probability at least $1-\frac12 e^{-z^2}$. 

To estimate \(A_k\), we split it into a particle-location error and an empirical-characteristic error:
\begin{align}
\label{eq:Ak-split}
A_k
&\leq
\E_X\max_l
\norm{\frac1N\sum_{i=1}^N
\left[
\psi_l(X;\omega_i^k)-\psi_l(X;\bar\omega_i(t_k))
\right]}
\nonumber\\
&\quad+
\E_X\max_l
\norm{\frac1N\sum_{i=1}^N\psi_l(X;\bar\omega_i(t_k))
-
\int\psi_l(X;\omega)\rho_{t_k}(\dd\omega)}.
\end{align}
The first term is bounded by \(C_K \max_{i\leq N}\norm{e_i^k}\), because \(\omega\mapsto\psi_l(X;\omega)\) is uniformly Lipschitz on \(K\).  The second term will be controlled by concentration for the independent characteristics.

Let
\begin{align*}
\bar\rho_k^N=\frac1N\sum_{i=1}^N\delta_{\bar\omega_i(t_k)}
\end{align*}
and define
\begin{align*}
B_k
=
\E_X\max_{l\leq L}
\norm{
\int\psi_l(X;\omega)(\bar\rho_k^N-\rho_{t_k})(\dd\omega)}.
\end{align*}
For each fixed \(k\), the variables \(\bar\omega_i(t_k)\) are independent with common law \(\rho_{t_k}\).  First, by Jensen's inequality and independence,
\begin{align*}
\E B_k
&\leq
\left(
\E\E_X\sum_{l=1}^L
\norm{
\frac1N\sum_{i=1}^N
\left[
\psi_l(X;\bar\omega_i(t_k))-
\int\psi_l(X;\omega)\rho_{t_k}(\dd\omega)
\right]
}^2
\right)^{1/2}  \\
&\leq
C B_\psi\sqrt{\frac{L}{N}}.
\end{align*}
Second, the map \((\bar\omega_1(t_k),\ldots,\bar\omega_N(t_k))\mapsto B_k\) has bounded differences: changing one coordinate changes \(B_k\) by at most \(C_K B_\psi/N\).  Hence McDiarmid's bounded-difference inequality gives
\begin{align*}
\Prob\left(B_k-\E B_k>s\right)
\leq
\exp\left(-\frac{Ns^2}{C_K^2 B_\psi^2}\right).
\end{align*}
Choosing
\begin{align}
\label{Choose_s}
s
=
C_K B_\psi N^{-1/2}
\sqrt{\log(\lfloor T/\varepsilon\rfloor+1)+z^2}
\end{align}
and taking a union bound over the grid \(k\in[0,T/\varepsilon]\cap\mathbb N\), we obtain, with probability at least \(1-e^{-z^2}/2\),
\begin{align}
\label{eq:empirical-characteristic-concentration}
\sup_{0\leq k\leq T/\varepsilon}B_k
\leq
C_KN^{-1/2}
\left[
\sqrt L+\sqrt{\log(\lfloor T/\varepsilon\rfloor+1)}+z
\right].
\end{align}

On the intersection of the two high-probability events \eqref{eq:martingale-concentration} and \eqref{eq:empirical-characteristic-concentration}, which has probability at least \(1-e^{-z^2}\), combine \eqref{eq:e-bound-before-gronwall}, \eqref{eq:martingale-concentration}, \eqref{eq:Ak-split} and \eqref{eq:empirical-characteristic-concentration} to get
\begin{align}
\label{eq:E-final-bound}
\sup_{0\leq k\leq T/\varepsilon}\max_{i\leq N}\norm{e_i^k}
\leq
C_T\operatorname{Err}_{N,\varepsilon,T}(z).
\end{align}

We now prove the test-function estimate \eqref{eq:w1-sgd-bound}.  Fix a function \(f:\R^{d_\omega}\to\R\) with \(\Lip(f)\leq1\).  Since constants cancel in empirical-measure differences, we may replace \(f\) by \(f-f(\omega_0)\) for some fixed \(\omega_0\in K\), and hence \(\abs f\leq \operatorname{diam}(K)\) on \(K\).  For every \(k\),
\begin{align*}
&\left|\frac1N\sum_{i=1}^N f(\omega_i^k)-\int f(\omega)\rho_{t_k}(\dd\omega)\right|\\
&\quad\leq
\frac1N\sum_{i=1}^N\abs{f(\omega_i^k)-f(\bar\omega_i(t_k))}
+
\left|\frac1N\sum_{i=1}^N f(\bar\omega_i(t_k))-\int f(\omega)\rho_{t_k}(\dd\omega)\right|\\
&\quad\leq
\max_{i\leq N}\norm{e_i^k}
+
\left|\frac1N\sum_{i=1}^N f(\bar\omega_i(t_k))-\int f(\omega)\rho_{t_k}(\dd\omega)\right|.
\end{align*}
For the second term, Hoeffding's inequality and a union bound over the grid \(k\in[0,T/\varepsilon]\cap\mathbb N\) give, with probability at least \(1-e^{-z^2}/2\),
\begin{align}
\label{eq:test-function-characteristic-concentration}
\sup_{0\leq k\leq T/\varepsilon}
\left|\frac1N\sum_{i=1}^N f(\bar\omega_i(t_k))-\int f(\omega)\rho_{t_k}(\dd\omega)\right|
\leq
C_TN^{-1/2}
\left[
\sqrt{\log(\lfloor T/\varepsilon\rfloor+1)}+z
\right].
\end{align}
Enlarging constants and intersecting this event with the event used above gives \eqref{eq:w1-sgd-bound}.  This final intersection only changes the numerical constant in the probability level, which is absorbed by replacing \(z\) by a comparable value or, equivalently, by increasing \(C_T\).

It remains to prove the risk estimate.  For each \(X\) and \(l\), the loss \(z\mapsto\ell(Y_l(X),z)\) is \(2\)-Lipschitz on \(\R^{d_v}\).  Therefore
\begin{align*}
\abs{\calE(\widehat\rho_k^N)-\calE(\rho_{t_k})}
&\leq
C\E_X\frac1L\sum_{l=1}^L
\norm{
\int\psi_l(X;\omega)(\widehat\rho_k^N-\rho_{t_k})(\dd\omega)}\\
&\leq
C_T \max_{i\leq N}\norm{e_i^k}
+
C_T N^{-1/2}
\left[
\sqrt{\log(\lfloor T/\varepsilon\rfloor+1)}+z
\right].
\end{align*}
The last line uses the same decomposition as in \eqref{eq:Ak-split} and the empirical-characteristic estimate \eqref{eq:empirical-characteristic-concentration}.  Taking the supremum over \(k\leq T/\varepsilon\) and using \eqref{eq:E-final-bound} proves \eqref{eq:risk-sgd-bound}, because \(\calE_N(\omega_1^k,\ldots,\omega_N^k)=\calE(\widehat\rho_k^N)\). This completes the proof of the theorem.
\end{proof}

\section{Long-time behavior of the limiting PDE}

This section studies the limiting PDE \eqref{eq:pde}. The central question is whether \(\rho_t\) converges to a stationary solution as \(t\to\infty\), and if so, at what rate.

\subsection{Energy dissipation and stationary solutions of the limiting PDE}
Define the dissipation functional
\begin{equation}
\label{eq:dissipation}
\mathcal I(\rho)
=
\int
\norm{\nabla_\omega\Phi(\omega;\rho)}^2
\rho(\dd\omega).
\end{equation}

\begin{theorem}[Energy dissipation]
\label{thm:energy-dissipation}
Let \(\rho_t\) solve \eqref{eq:pde}. Then
\begin{equation}
\label{eq:energy-dissipation}
\frac{\dd}{\dd t}\calE(\rho_t)
=
-2\zeta(t)\mathcal I(\rho_t)
\leq 0.
\end{equation}
Consequently, \(\calE(\rho_t)\) is non-increasing in time.
\end{theorem}

\begin{remark}[Gradient-flow interpretation]
This identity is the direct analogue of the energy dissipation equality for Wasserstein gradient flows in \citet{AmbrosioGigliSavare2008} and for mean-field neural-network training in optimal-transport form, such as \citet{ChizatBach2018}. The only model-specific ingredient is the first variation \(\Phi\), which encodes the self-attention logits and the cross-entropy loss. The result shows that the limiting PDE is genuinely a risk-decreasing dynamics, not merely a formal large-head limit.
\end{remark}

\begin{proof}
Using the first variation and the weak formulation,
\begin{align*}
\frac{\dd}{\dd t}\calE(\rho_t)
=
\int
\Phi(\omega;\rho_t)\partial_t\rho_t(\dd\omega).
\end{align*}
Substituting \eqref{eq:pde} and integrating by parts gives
\begin{align*}
\frac{\dd}{\dd t}\calE(\rho_t)
=
-2\zeta(t)
\int
\norm{\nabla_\omega\Phi(\omega;\rho_t)}^2
\rho_t(\dd\omega),
\end{align*}
which proves the claim.
\end{proof}

\begin{theorem}[Stationary solutions]
\label{thm:stationary}
A measure \(\bar\rho\in\calP(K)\) is a stationary solution of \eqref{eq:pde} if and only if
\begin{equation}
\label{eq:stationary-condition}
\nabla_\omega\Phi(\omega;\bar\rho)=0
\quad
\bar\rho\text{-almost surely}.
\end{equation}
If \(\nabla_\omega\Phi(\cdot;\bar\rho)\) is continuous, this is equivalent to
\begin{equation}
\label{eq:stationary-support}
\supp(\bar\rho)
\subset
\left\{
\omega\in K:
\nabla_\omega\Phi(\omega;\bar\rho)=0
\right\}.
\end{equation}
\end{theorem}

\begin{remark}[Stationarity versus optimality]
The condition \(\nabla_\omega\Phi=0\) on the support is the natural equilibrium condition for the transport PDE. Similar distinctions between stationary measures and global minimizers appear in mean-field neural-network analyses such as \citet{MeiMontanariNguyen2018} and \citet{ChizatBach2018}. In the present attention model this distinction is important: \Cref{thm:minimizer-support} requires minimization of \(\Phi\) over all \(K\), whereas stationarity only requires vanishing of its gradient on the current support.
\end{remark}

\begin{proof}
If \(\bar\rho\) is stationary, then \(\rho_t\equiv\bar\rho\). By \Cref{thm:energy-dissipation},
\begin{align*}
0
=
\frac{\dd}{\dd t}\calE(\rho_t)
=
-2\zeta(t)
\int
\norm{\nabla_\omega\Phi(\omega;\bar\rho)}^2
\bar\rho(\dd\omega).
\end{align*}
Since \(\zeta(t)\geq0\) and is not identically zero, this implies \eqref{eq:stationary-condition}. Conversely, if \eqref{eq:stationary-condition} holds, then the right-hand side of the weak formulation \eqref{eq:weak-pde} vanishes for every test function, and therefore \(\rho_t\equiv\bar\rho\) is a weak solution.
\end{proof}

\begin{remark}
Every interior global minimizer satisfying the support condition in \Cref{thm:minimizer-support} is stationary, because differentiability of \(\Phi(\cdot;\rho_\ast)\) implies that its gradient vanishes on the interior of the set where it attains its minimum. The converse is false in general: stationary solutions may be local minimizers, saddle-type stationary measures, or even unstable equilibria.
\end{remark}

\subsection{Long-time convergence of the limiting PDE}

Define the rescaled time
\begin{equation}
\label{eq:clock}
\clock(t)
=
2\int_0^t\zeta(r)\dd r.
\end{equation}
By Assumption \ref{ass:learning-rate}, \(\clock(t)\to\infty\) as \(t\to\infty\). In the \(\clock\)-time variable, the PDE becomes
\begin{align*}
\partial_{\clock}\rho_{\clock}
=
\nabla_\omega\cdot
\left(
\rho_{\clock}\nabla_\omega\Phi(\omega;\rho_{\clock})
\right),
\end{align*}
and
\begin{align*}
\frac{\dd}{\dd \clock}\calE(\rho_{\clock})
=
-\mathcal I(\rho_{\clock}).
\end{align*}

\begin{theorem}[Convergence to the stationary set]
\label{thm:omega-limit}
Assume Assumption \ref{ass:smoothness} and suppose that the trajectory \(\{\rho_t:t\geq0\}\) is precompact in \(W_1\). Then:
\begin{enumerate}
\item The limit
\begin{align*}
\calE_\infty
=
\lim_{t\to\infty}\calE(\rho_t)
\end{align*}
exists.
\item \(\mathcal I(\rho_t)\to0\) as \(t\to\infty\).
\item Every \(W_1\)-limit point of \(\rho_t\) is a stationary solution of \eqref{eq:pde}.
\item The omega-limit set
\begin{align*}
\Omega(\rho_0)
=
\left\{
\bar\rho:
\exists\,t_n\to\infty
\text{ such that }
W_1(\rho_{t_n},\bar\rho)\to0
\right\}
\end{align*}
is nonempty, compact, connected, and contained in the stationary set.
\item If the stationary set intersected with \(\Omega(\rho_0)\) is totally disconnected, in particular if it is finite, then \(\rho_t\) converges in \(W_1\) to a single stationary solution.
\end{enumerate}
\end{theorem}

\begin{remark}[Checkable criteria for precompactness]
The precompactness assumption in \Cref{thm:omega-limit} can be checked at the level of the characteristic equation.  In the rescaled time variable \(\clock\), a characteristic satisfies
\begin{align*}
\frac{\dd}{\dd\clock}W_\clock
=
-\nabla_\omega\Phi(W_\clock;\rho_\clock),
\end{align*}
where
\begin{align*}
\nabla_\omega\Phi(\omega;\rho)
=
\E_X\left[
\frac1L\sum_{l=1}^L
\nabla_\omega\psi_l(X;\omega)^\top
\bigl(\softmax(z_l^\rho(X))-Y_l(X)\bigr)
\right].
\end{align*}
Since \(\norm{\softmax(z_l^\rho(X))-Y_l(X)}\leq2\), the size and direction of the vector field are controlled by the derivatives of the attention feature maps.  A simple sufficient condition is an invariant compact set: if \(\supp\rho_0\subset K_R:=\{\omega:\norm{\omega-\omega_c}\leq R\}\) and
\begin{align*}
\inner{\omega-\omega_c}{\nabla_\omega\Phi(\omega;\rho)}\geq0
\qquad
\text{for every }\omega\in\partial K_R
\end{align*}
whenever \(\rho\in\calP(K_R)\), then
\begin{align*}
\frac{\dd}{\dd\clock}\norm{W_\clock-\omega_c}^2
=
-2\inner{W_\clock-\omega_c}{\nabla_\omega\Phi(W_\clock;\rho_\clock)}\leq0
\end{align*}
on the boundary, and hence the support remains in \(K_R\).  This immediately gives precompactness in \(W_1\).  More generally, a dissipative confinement condition such as
\begin{align*}
\inner{\omega}{\nabla_\omega\Phi(\omega;\rho)}
\geq
 a\norm{\omega}^2-b,
\qquad a>0,
\end{align*}
for all sufficiently large \(\norm\omega\) and all relevant \(\rho\), yields a uniform second-moment estimate by differentiating \(\E\norm{W_\clock}^2\).  Together with tightness on finite-dimensional parameter space, this gives relative compactness in \(W_1\).  These conditions are directly verifiable from the displayed formula for \(\nabla_\omega\Phi\) once growth or inward-pointing bounds for \(\nabla_\omega\psi_l\) are available.  Without such an invariant-domain or confinement mechanism, boundedness of \(\nabla_\omega\psi_l\) alone gives finite-speed motion on finite time intervals but does not by itself imply long-time precompactness.
\end{remark}

\begin{remark}[Comparison with general gradient systems]
This theorem plays the role of a LaSalle-type invariance result for the Wasserstein gradient flow. It is weaker than global convergence theorems that rely on displacement convexity or a global Polyak-{\L}ojasiewicz inequality, but it is stronger than finite-time propagation-of-chaos estimates because it identifies the asymptotic stationary set of the limiting PDE. The statement follows the general philosophy of metric gradient flows in \citet{AmbrosioGigliSavare2008}, while keeping the assumptions local enough to apply to the nonconvex attention-head parameterization.
\end{remark}

\begin{proof}
Since \(\calE(\rho_t)\) is non-increasing and bounded below by \(0\), the limit \(\calE_\infty\) exists. In the rescaled time variable,
\begin{align*}
\int_0^\infty
\mathcal I(\rho_{\clock})
\dd\clock
=
\calE(\rho_0)-\calE_\infty
<\infty.
\end{align*}
On a compact parameter set, \(\rho\mapsto \mathcal I(\rho)\) is continuous with respect to \(W_1\), and the trajectory is uniformly continuous in \(W_1\). Hence \(\clock\mapsto \mathcal I(\rho_{\clock})\) is uniformly continuous. Since it is nonnegative and integrable over \([0,\infty)\), Barbalat's lemma implies 
% This lemma says that if a function \(g:[0,\infty)\to[0,\infty)\) is uniformly continuous and \(\int_0^\infty g(s)\dd s<\infty\), then \(g(s)\to0\). Indeed, if \(g(s_n)\geq\eta>0\) for some \(s_n\to\infty\), uniform continuity gives a length \(\delta>0\) such that \(g(s)\geq\eta/2\) on \([s_n-\delta,s_n+\delta]\) for infinitely many disjoint intervals, contradicting integrability. 
%Applying this lemma to \(g(\clock)=\mathcal I(\rho_\clock)\) gives
\begin{align*}
\mathcal I(\rho_{\clock})\to0
\end{align*}
as \(\clock\to\infty\). Returning to physical time gives the same conclusion because \(\clock(t)\to\infty\) as \(t\to\infty\).

Let \(\rho_{t_n}\to\bar\rho\) in \(W_1\). By continuity of \(\mathcal I\),
\begin{align*}
\mathcal I(\bar\rho)=0.
\end{align*}
Therefore
\begin{align*}
\nabla_\omega\Phi(\omega;\bar\rho)=0
\quad
\bar\rho\text{-almost surely},
\end{align*}
so \(\bar\rho\) is stationary by \Cref{thm:stationary}.

Precompactness implies that for every sequence \(t_n\to\infty\) there is a subsequence \(t_{n_j}\) and a measure \(\bar\rho\) such that \(W_1(\rho_{t_{n_j}},\bar\rho)\to0\); hence \(\Omega(\rho_0)\) is nonempty. Moreover,
\begin{align*}
\Omega(\rho_0)
=
\bigcap_{T>0}
\overline{\{\rho_t:t\geq T\}}^{W_1},
\end{align*}
where each tail closure is compact. Indeed, by precompactness the closure of \(\{\rho_t:t\geq0\}\) is compact in \(W_1\), and each set \(\overline{\{\rho_t:t\geq T\}}^{W_1}\) is a closed subset of this compact set. The family is nested: if \(T_2\geq T_1\), then
\begin{align*}
\overline{\{\rho_t:t\geq T_2\}}^{W_1}
\subset
\overline{\{\rho_t:t\geq T_1\}}^{W_1}.
\end{align*}
Therefore their intersection is closed inside the compact set \(\overline{\{\rho_t:t\geq0\}}^{W_1}\), hence compact.

We next prove connectedness. First note that
\begin{align}
\label{eq:distance-to-omega-limit}
\operatorname{dist}_{W_1}(\rho_t,\Omega(\rho_0))\to0.
\end{align}
Otherwise there would exist \(\eta>0\) and \(t_n\to\infty\) such that \(\operatorname{dist}_{W_1}(\rho_{t_n},\Omega(\rho_0))\geq\eta\). By precompactness, a subsequence of \(\rho_{t_n}\) converges to some element of \(\Omega(\rho_0)\), a contradiction. Suppose, for contradiction, that \(\Omega(\rho_0)\) is not connected. Then there exist disjoint nonempty compact sets \(A,B\) such that
\begin{align*}
\Omega(\rho_0)\subset A\cup B,
\qquad
\operatorname{dist}_{W_1}(A,B)=\delta>0.
\end{align*}
By definition of \(\Omega(\rho_0)\), the trajectory visits the \(\delta/4\)-neighborhoods of both \(A\) and \(B\) at arbitrarily large times. Since \(t\mapsto\rho_t\) is continuous in \(W_1\), between two such visits it must pass through the closed set
\begin{align*}
\Gamma=
\left\{\rho:
\operatorname{dist}_{W_1}(\rho,A)\geq\frac{\delta}{4},
\quad
\operatorname{dist}_{W_1}(\rho,B)\geq\frac{\delta}{4}
\right\}.
\end{align*}
Thus there are times \(s_n\to\infty\) with \(\rho_{s_n}\in\Gamma\). A convergent subsequence has a limit point \(\rho_\infty\in\Omega(\rho_0)\cap\Gamma\), contradicting \(\Omega(\rho_0)\subset A\cup B\) and the definition of \(\Gamma\). Hence \(\Omega(\rho_0)\) is connected.

If the stationary set intersected with \(\Omega(\rho_0)\) is totally disconnected, then the connected set \(\Omega(\rho_0)\), which is contained in that stationary set, must consist of a single point, say \(\bar\rho\). Finally, if \(\rho_t\) did not converge to \(\bar\rho\), then there would be \(\eta>0\) and \(t_n\to\infty\) such that \(W_1(\rho_{t_n},\bar\rho)\geq\eta\). A convergent subsequence would produce a second point of \(\Omega(\rho_0)\), contradicting \(\Omega(\rho_0)=\{\bar\rho\}\). Therefore \(\rho_t\to\bar\rho\) in \(W_1\).
\end{proof}

%\subsection{Rates from gradient domination and {\L}ojasiewicz inequalities}

\begin{proposition}[Precompactness of compactly supported trajectories]
\label{prop:w1-precompactness}
Assume Assumption \ref{ass:smoothness} and let \(\rho_t\) be a weak solution of \eqref{eq:pde} such that
\begin{align*}
\supp\rho_t\subset K
\qquad
\text{for every }t\geq0.
\end{align*}
Then the trajectory \(\{\rho_t:t\geq0\}\) is precompact in \(W_1\). In particular, the precompactness hypothesis in \Cref{thm:omega-limit} is automatic for the compactly supported solutions considered under Assumption \ref{ass:smoothness}.
\end{proposition}

\begin{proof}
Since \(K\) is compact, every probability measure supported in \(K\) has a uniformly bounded first moment. Moreover, narrow convergence on \(\calP(K)\) is equivalent to convergence in \(W_1\). Indeed, if \(\mu_n\) converges narrowly to \(\mu\) and all measures are supported in \(K\), then for a fixed \(\omega_0\in K\) the function \(\omega\mapsto\norm{\omega-\omega_0}\) is bounded and continuous on \(K\), so the first moments converge automatically. The standard characterization of \(W_1\)-convergence as narrow convergence plus convergence of first moments therefore gives \(W_1(\mu_n,\mu)\to0\). Conversely, \(W_1\)-convergence implies narrow convergence by the Kantorovich--Rubinstein duality, since every bounded Lipschitz test function is continuous and has controlled Lipschitz seminorm on \(K\). By Prokhorov's theorem, compactness of \(K\) implies narrow compactness of \(\calP(K)\); and hence \(\calP(K)\) compact in \(W_1\). Therefore any sequence \(\rho_{t_n}\) has a subsequence converging in \(W_1\) to a probability measure supported in \(K\). This proves precompactness.
\end{proof}

The previous theorem gives convergence to the stationary set but does not give a rate. Rates require additional information on the geometry of the risk near the limiting stationary set.

\begin{theorem}[Kurdyka--{\L}ojasiewicz rates]
\label{thm:kl-rate}
Let \(\bar\rho\) be a stationary solution of \eqref{eq:pde}. Suppose that for some neighborhood \(\mathcal U\) of \(\bar\rho\), constants \(c_{\mathrm{KL}}>0\), \(\theta\in[1/2,1)\), and all \(\rho\in\mathcal U\),
\begin{equation}
\label{eq:kl}
\mathcal I(\rho)^{1/2}
\geq
c_{\mathrm{KL}}
\bigl(
\calE(\rho)-\calE(\bar\rho)
\bigr)^\theta.
\end{equation}
Assume that \(\rho_t\in\mathcal U\) for \(t\geq t_0\). Let
\begin{align*}
G(t)=\calE(\rho_t)-\calE(\bar\rho).
\end{align*}
Then:
\begin{enumerate}
\item If \(\theta=1/2\), then
\begin{align*}
G(t)
\leq
G(t_0)
\exp\left(
-2c_{\mathrm{KL}}^2
\int_{t_0}^t\zeta(r)\dd r
\right).
\end{align*}
\item If \(1/2<\theta<1\), then
\begin{align*}
G(t)
\leq
\left[
G(t_0)^{1-2\theta}
+
2c_{\mathrm{KL}}^2(2\theta-1)
\int_{t_0}^t\zeta(r)\dd r
\right]^{-1/(2\theta-1)}.
\end{align*}
\end{enumerate}
\end{theorem}

\begin{remark}[Checkable sources of the KL inequality]
The KL inequality in \Cref{thm:kl-rate} is an additional landscape assumption, but there are several standard situations in which it can be verified from the structure of \(\nabla_\omega\Phi\).  First, suppose that the dynamics is restricted to a finite-atomic ansatz
\begin{align*}
\rho=\sum_{q=1}^Q a_q\delta_{\omega_q},
\qquad
 a_q\geq0,
\qquad
\sum_{q=1}^Q a_q=1.
\end{align*}
If \(\omega\mapsto\psi_l(X;\omega)\) is real analytic on a neighborhood of the relevant compact set and the expectation in \(X\) preserves analyticity, for instance for a finite empirical data distribution, then \(\calE\) restricted to the variables \((a_q,\omega_q)_{q=1}^Q\) is a finite-dimensional real-analytic function.  The classical finite-dimensional Kurdyka--{\L}ojasiewicz theorem then gives a local KL inequality near every critical point of this finite-dimensional reduction.

A second, more quantitative, sufficient condition is a local Polyak--{\L}ojasiewicz estimate.  If, in a neighborhood of \(\bar\rho\), one can show
\begin{align*}
\mathcal I(\rho)
=
\int\norm{\nabla_\omega\Phi(\omega;\rho)}^2\rho(\dd\omega)
\geq
\lambda\bigl(\calE(\rho)-\calE(\bar\rho)\bigr)
\end{align*}
for some \(\lambda>0\), then \eqref{eq:kl} holds with \(\theta=1/2\) and \(c_{\mathrm{KL}}=\sqrt\lambda\).  Such an estimate can often be checked by combining the displayed expression for \(\nabla_\omega\Phi\) with a nondegeneracy condition on the feature Jacobians \(\nabla_\omega\psi_l\) and a positive lower bound for the softmax Hessian \(\diag(p_l)-p_lp_l^\top\) on the identifiable logit subspace.  Thus the KL assumption may be viewed as a local observability condition: the residual direction in logit space must be detected, with a uniform lower bound, by the gradients of the attention-head features.
\end{remark}

\begin{remark}
If \(\zeta(t)\geq\zeta_->0\), then the rates in \Cref{thm:kl-rate} become rates in physical time. For \(\theta=1/2\),
\begin{align*}
G(t)\leq G(t_0)\exp\left(-2c_{\mathrm{KL}}^2\zeta_-(t-t_0)\right).
\end{align*}
For \(1/2<\theta<1\),
\begin{align*}
G(t)
\leq
\left[
G(t_0)^{1-2\theta}
+
2c_{\mathrm{KL}}^2(2\theta-1)\zeta_-(t-t_0)
\right]^{-1/(2\theta-1)},
\end{align*}
so the energy gap is algebraic of order \(t^{-1/(2\theta-1)}\).
\end{remark}

\begin{remark}[Comparison with finite-dimensional {\L}ojasiewicz theory]
The theorem is the measure-space counterpart of the standard Kurdyka-{\L}ojasiewicz mechanism for finite-dimensional gradient flows. The paper does not claim that the KL inequality automatically holds for the full attention mean-field functional; rather, it isolates the exact inequality that converts energy dissipation into rates. This is analogous to how mean-field analyses of neural networks often separate the derivation of the limiting dynamics from additional landscape assumptions needed for convergence rates.
\end{remark}

\begin{proof}
By \Cref{thm:energy-dissipation} and \eqref{eq:kl},
\begin{align*}
G'(t)
=
-2\zeta(t)\mathcal I(\rho_t)
\leq
-2c_{\mathrm{KL}}^2\zeta(t)G(t)^{2\theta}.
\end{align*}
The case \(\theta=1/2\) follows directly from Gronwall's inequality. For \(1/2<\theta<1\), integrate the differential inequality for \(G(t)^{1-2\theta}\).
\end{proof}

\section{Stability of stationary solutions}

We next analyze stability of stationary solutions. The results in this section also give convergence rates in Wasserstein distance.

\subsection{A general strong-monotonicity criterion}

\begin{theorem}[Local Wasserstein exponential stability]
\label{thm:wasserstein-stability}
Let \(\bar\rho\) be a stationary solution of \eqref{eq:pde}. Assume that there exist \(r>0\) and \(\kappa>0\) such that for every \(\rho\) satisfying \(W_2(\rho,\bar\rho)<r\) and every coupling \(\pi\in\Pi(\rho,\bar\rho)\),
\begin{equation}
\label{eq:strong-monotonicity}
\int
\inner{\omega-\omega'}{
\nabla_\omega\Phi(\omega;\rho)
-
\nabla_\omega\Phi(\omega';\bar\rho)
}
\pi(\dd \omega,\dd \omega')
\geq
\kappa
\int\norm{\omega-\omega'}^2\pi(\dd \omega,\dd \omega').
\end{equation}
If \(\rho_t\) is a solution with \(W_2(\rho_0,\bar\rho)<r\) and the trajectory remains in this \(W_2\)-ball, then
\begin{equation}
\label{eq:w2-exp-stability}
W_2(\rho_t,\bar\rho)^2
\leq
W_2(\rho_0,\bar\rho)^2
\exp\left(
-4\kappa\int_0^t\zeta(s)\dd s
\right).
\end{equation}
Consequently,
\begin{align*}
W_2(\rho_t,\bar\rho)
\leq
W_2(\rho_0,\bar\rho)
\exp\left(
-2\kappa\int_0^t\zeta(s)\dd s
\right).
\end{align*}
If \(\zeta(t)\geq\zeta_->0\), the convergence is exponential in physical time.
\end{theorem}

\begin{remark}[A checkable sufficient condition for strong monotonicity]
The assumption in \Cref{thm:wasserstein-stability} can be verified by separating the local curvature in \(\omega\) from the measure-dependence of the vector field.  Suppose that there is a neighborhood \(U\) of \(\supp\bar\rho\) and constants \(m>0\), \(\gamma\geq0\) such that, for all \(\omega,\omega'\in U\) and all \(\rho\) supported in \(U\),
\begin{align*}
\inner{\omega-\omega'}{\nabla_\omega\Phi(\omega;\bar\rho)-\nabla_\omega\Phi(\omega';\bar\rho)}
&\geq
m\norm{\omega-\omega'}^2,\\
\norm{\nabla_\omega\Phi(\omega;\rho)-\nabla_\omega\Phi(\omega;\bar\rho)}
&\leq
\gamma W_2(\rho,\bar\rho).
\end{align*}
The first inequality follows, for example, if \(\nabla_\omega^2\Phi(\omega;\bar\rho)\succeq mI\) on \(U\).  The second one follows from the formula for \(\nabla_\omega\Phi\), the Lipschitz continuity of softmax, and the Lipschitz bound on \(\omega\mapsto\psi_l(X;\omega)\): indeed,
\begin{align*}
\norm{\nabla_\omega\Phi(\omega;\rho)-\nabla_\omega\Phi(\omega;\bar\rho)}
\leq
C_K\E_X\max_l
\norm{\int\psi_l(X;\omega)(\rho-\bar\rho)(\dd\omega)}
\leq
C_K W_2(\rho,\bar\rho),
\end{align*}
so one may take \(\gamma=C_K\) when the displayed bound is valid on \(U\).  For any coupling \(\pi\in\Pi(\rho,\bar\rho)\), Cauchy's inequality and \(W_2(\rho,\bar\rho)^2\leq\int\norm{\omega-\omega'}^2\pi(\dd \omega,\dd \omega')\) give
\begin{align*}
&\int\inner{\omega-\omega'}{\nabla_\omega\Phi(\omega;\rho)-\nabla_\omega\Phi(\omega';\bar\rho)}\pi(\dd \omega,\dd \omega')\\
&\qquad\geq
m\int\norm{\omega-\omega'}^2\pi(\dd \omega,\dd \omega')
-
\gamma W_2(\rho,\bar\rho)
\left(\int\norm{\omega-\omega'}^2\pi(\dd \omega,\dd \omega')\right)^{1/2}\\
&\qquad\geq
(m-\gamma)\int\norm{\omega-\omega'}^2\pi(\dd \omega,\dd \omega').
\end{align*}
Hence the strong-monotonicity condition holds with any \(\kappa<m-\gamma\), provided \(m>\gamma\) and the trajectory remains in the neighborhood where the estimates are valid.  This criterion is particularly transparent for Dirac or finitely atomic stationary measures, where the curvature condition \eqref{eq:strong-monotonicity} reduces to positive definiteness of the corresponding Hessian matrices \eqref{eq:H0-positive}; see the following \Cref{thm:dirac-stability}.
\end{remark}

\begin{remark}[Relation to contractivity of Wasserstein gradient flows]
The strong-monotonicity assumption is a local version of the monotonicity or convexity conditions that yield contraction estimates for Wasserstein gradient flows, as in \citet{AmbrosioGigliSavare2008}. Unlike global displacement convexity, which is generally unavailable for neural-network parameterizations, this theorem only asks for monotonicity near a chosen stationary measure. It therefore gives a stability criterion tailored to nonconvex mean-field models.
\end{remark}

\begin{remark}[Lions-lift interpretation of strong monotonicity]
The monotonicity condition \eqref{eq:strong-monotonicity} can also be interpreted through the Lions lift of the mean-field risk.  Let
\begin{align*}
\widetilde{\calE}(W)=\calE(\calL(W)),
\qquad
W\in L^2(\Omega,\mathcal F,\mathbb P;\R^{d_\omega}),
\end{align*}
where \(\calL(W)\) denotes the law of \(W\).  The Fr\'echet derivative of this lifted functional, $D\widetilde{\calE}(W)$, is the Lions derivative of \(\calE(\rho)\) evaluated at $\rho=\calL(W)$, and in the present model it is given by
\begin{align*}
D\widetilde{\calE}(W)=\nabla_\omega\Phi(W;\calL(W))
\quad\text{in }L^2(\Omega,\mathcal F,\mathbb P;\R^{d_\omega}).
\end{align*}
If \((W,\bar W)\) is a coupling of \(\rho\) and \(\bar\rho\), then \eqref{eq:strong-monotonicity} is exactly the lifted strong monotonicity estimate
\begin{align}
\label{eq:lions-lift-monotonicity}
D\widetilde{\calE}(W)[W_\delta]-D\widetilde{\calE}(\bar W)[W_\delta]=\E\inner{W-\bar W}{D\widetilde{\calE}(W)-D\widetilde{\calE}(\bar W)}
\geq
\kappa\E\norm{W-\bar W}^2=\kappa\E\norm{W_\delta}^2,
\end{align}
where \(W_\delta=W-\bar W\). When \(\bar\rho\) is stationary, one may choose \(\bar W\) with law \(\bar\rho\), and then \(D\widetilde{\calE}(\bar W)=\nabla_\omega\Phi(\bar W;\bar\rho)=0\) almost surely.  Thus \eqref{eq:strong-monotonicity} says that the lifted gradient of the mean-field risk is locally strongly monotone around the stationary lift \(\bar W\).  Equivalently, along line segments in the lifted Hilbert space on which \eqref{eq:lions-lift-monotonicity} holds, the lifted risk behaves like a locally strongly convex functional.

If, in addition, \(\widetilde{\calE}\) is twice Fr\'echet differentiable at \(\bar W\) and \(\calE(\rho)\) is twice regularly differentiable with respect to $\rho\in\mathcal{P}_2(K)$ in the sense of Definition~3.2 of \citep{BWYY2026}, then the infinitesimal version of \eqref{eq:lions-lift-monotonicity} is
\begin{align}
\label{eq:lions-lift-hessian-coercivity}
D^2\widetilde{\calE}(\bar W)[W_\delta,W_\delta]
\geq
\kappa\E\norm{W_\delta}^2,
\qquad
W_\delta\in L^2(\Omega,\mathcal F,\mathbb P;\R^{d_\omega}).
\end{align}
Letting \((\widetilde {\bar W},\widetilde W_\delta)\) be an independent copy of \((\bar W,W_\delta)\), the second-order formula for the Lions lift, corresponding to equation~(3.3) in \citep{BWYY2026}, gives
\begin{align}
\label{eq:lions-lift-second-variation}
D^2\widetilde{\calE}(\bar W)[W_\delta,W_\delta]
=&\ \E\left[
W_\delta^\top
\nabla_\omega^2\frac{\delta\calE}{\delta\rho}(\bar\rho)(\bar W)
W_\delta
\right]
\nonumber\\
&+\E\widetilde\E\left[
W_\delta^\top
\nabla_\omega\nabla_{\widetilde\omega}
\frac{\delta^2\calE}{\delta\rho^2}(\bar\rho)(\bar W,\widetilde {\bar W})
\widetilde W_\delta
\right].
\end{align}
The first term measures the curvature of the first variation in the particle variable, while the second term measures the curvature coming from the dependence of the first variation on the law.  Therefore the local Wasserstein strong monotonicity condition used in \Cref{thm:wasserstein-stability} is the transport analogue of a positive-definiteness condition for the Hessian of the Lions lift.  Monotonicity and convexity assumptions of this type are also central in mean-field type control problems and mean field games; see \citet{BWYY2026} for a systematic treatment.
\end{remark}

\begin{proof}
Let \((W_0,\bar W_0)\) be an optimal coupling of \(\rho_0\) and \(\bar\rho\). Let \(W_t\) solve
\begin{align*}
\dot W_t
=
-2\zeta(t)\nabla_\omega\Phi(W_t;\rho_t),
\qquad
\calL(W_t)=\rho_t.
\end{align*}
Since \(\bar\rho\) is stationary, \(\nabla_\omega\Phi(\bar W_0;\bar\rho)=0\) for \(\bar\rho\)-almost every \(\bar W_0\). Hence \(\bar W_t=\bar W_0\) is a stationary characteristic associated with \(\bar\rho\). Then \(\pi_t=\calL(W_t,\bar W_0)\) is a coupling of \(\rho_t\) and \(\bar\rho\). Define
\begin{align*}
D(t)=\E\norm{W_t-\bar W_0}^2.
\end{align*}
Differentiating,
\begin{align*}
D'(t)
=
-4\zeta(t)
\E
\inner{W_t-\bar W_0}{
\nabla_\omega\Phi(W_t;\rho_t)
-
\nabla_\omega\Phi(\bar W_0;\bar\rho)
}.
\end{align*}
By \eqref{eq:strong-monotonicity},
\begin{align*}
D'(t)
\leq
-4\kappa\zeta(t)D(t).
\end{align*}
Therefore
\begin{align*}
D(t)
\leq
D(0)
\exp\left(
-4\kappa\int_0^t\zeta(s)\dd s
\right).
\end{align*}
Since \(W_2(\rho_t,\bar\rho)^2\leq D(t)\) and \(D(0)=W_2(\rho_0,\bar\rho)^2\), the result follows.
\end{proof}

\subsection{Local stability of Dirac stationary measures}

We now give a concrete stability result for single-point stationary measures. Let
\begin{align*}
\rho_\ast=\delta_{\omega_\ast}.
\end{align*}
Define
\begin{equation}
\label{eq:H0}
H_0
=
\nabla_\omega^2\Phi(\omega_\ast;\rho_\ast).
\end{equation}

\begin{theorem}[Dirac stationarity and local exponential stability]
\label{thm:dirac-stability}
Assume Assumption \ref{ass:smoothness}. Let \(\rho_\ast=\delta_{\omega_\ast}\). Then \(\rho_\ast\) is a stationary solution of \eqref{eq:pde} if and only if
\begin{equation}
\label{eq:dirac-stationary}
\nabla_\omega\Phi(\omega_\ast;\rho_\ast)=0.
\end{equation}
Assume in addition that
\begin{equation}
\label{eq:H0-positive}
H_0\succeq \lambda_0 I
\qquad
\text{for some }\lambda_0>0.
\end{equation}
Then there exist \(r_0>0\) and \(c_\ast>0\) such that if
\begin{align*}
\supp(\rho_0)\subset B(\omega_\ast;r_0),
\end{align*}
then the solution remains in a small neighborhood of \(\omega_\ast\) for all \(t\geq0\) and
\begin{equation}
\label{eq:dirac-exp-rate}
W_2(\rho_t,\rho_\ast)^2
=
\int\norm{\omega-\omega_\ast}^2\rho_t(\dd\omega)
\leq
\exp\left(
-4c_\ast\int_0^t\zeta(s)\dd s
\right)
\int\norm{\omega-\omega_\ast}^2\rho_0(\dd\omega).
\end{equation}
% Here the identity
% \(W_2(\rho_t,\rho_\ast)^2=
% \int\norm{\omega-\omega_\ast}^2\rho_t(\dd\omega)\)
% follows from the fact that \(\rho_\ast\) is a Dirac mass: the only coupling of \(\rho_t\) with \(\delta_{\omega_\ast}\) is
% \(\pi(\dd\omega,\dd\eta)=\rho_t(\dd\omega)\delta_{\omega_\ast}(\dd\eta)\), and therefore the infimum in the definition of \(W_2\) is attained by this coupling and equals the displayed second moment.
In particular, if \(\zeta(t)\geq\zeta_->0\), then \(\rho_t\to\delta_{\omega_\ast}\) exponentially fast in \(W_2\).
\end{theorem}

\begin{remark}[Comparison with particle-level stability]
Dirac stationary measures correspond to synchronized particle configurations in which all heads sit at the same parameter. The theorem is analogous to a local Hessian stability test for an equilibrium of a finite-dimensional gradient system, but the proof must also control the measure argument in \(\Phi(\omega;\rho)\). This additional measure dependence is absent in ordinary single-particle gradient descent and is the main reason for the moment and support bootstrap estimates.
\end{remark}

\begin{proof}
The stationarity assertion follows directly from \Cref{thm:stationary}.

We prove the stability statement.  For each parameter value \(\omega\), write
\begin{align*}
\omega_\delta=\omega-\omega_\ast,
\qquad
M_2(\rho)=\int\norm{\omega_\delta}^2\rho(\dd\omega).
\end{align*}
We first establish a local coercivity estimate. Since \(H_0\succeq\lambda_0I\) and \(\nabla_\omega^2\Phi\) is continuous in both \(\omega\) and \(\rho\), there exists \(r>0\) such that, whenever \(\supp(\rho)\subset B(\omega_\ast;r)\),
\begin{equation}
\label{eq:hessian-local}
\nabla_\omega^2\Phi(\omega;\rho)
\succeq
\frac{3\lambda_0}{4}I
\qquad
\text{for all }\omega\in B(\omega_\ast;r).
\end{equation}
The dependence of \(\nabla_\omega^2\Phi\) on \(\rho\) is controlled by the Lipschitz continuity of softmax and the \(C^2\)-bounds on \(\psi_l\).

For such \(\rho\), decompose
\begin{align*}
\int
\inner{\omega_\delta}{\nabla_\omega\Phi(\omega;\rho)}
\rho(\dd\omega)
=
A(\rho)+B(\rho),
\end{align*}
where
\begin{align*}
A(\rho)
=
\int
\inner{\omega_\delta}{
\nabla_\omega\Phi(\omega;\rho)
-
\nabla_\omega\Phi(\omega_\ast;\rho)
}
\rho(\dd\omega)
\end{align*}
and
\begin{align*}
B(\rho)
=
\int
\inner{\omega_\delta}{
\nabla_\omega\Phi(\omega_\ast;\rho)
}
\rho(\dd\omega).
\end{align*}
By \eqref{eq:hessian-local},
\begin{align}\label{Ineq:A}
A(\rho)
\geq
\frac{3\lambda_0}{4}M_2(\rho).
\end{align}

It remains to lower bound \(B(\rho)\). Let
\begin{align*}
m=\int \omega_\delta\rho(\dd\omega).
\end{align*}
For each \(X,l\), set
\begin{align*}
a_{l,X}(\omega)=\psi_l(X;\omega),
\qquad
a_{l,\ast}=a_{l,X}(\omega_\ast),
\qquad
J_{l,\ast}=\nabla_\omega a_{l,X}(\omega_\ast).
\end{align*}
Since \(\nabla_\omega\Phi(\omega_\ast;\rho_\ast)=0\),
\begin{align*}
\nabla_\omega\Phi(\omega_\ast;\rho)
=
\E_X
\left[
\frac{1}{L}
\sum_{l=1}^L
J_{l,\ast}^\top
\left(
\softmax(z_l^\rho(X))-\softmax(a_{l,\ast})
\right)
\right].
\end{align*}
Thus
\begin{align*}
B(\rho)
=
\E_X
\left[
\frac{1}{L}
\sum_{l=1}^L
\inner{
J_{l,\ast}m
}{
\softmax(z_l^\rho(X))-\softmax(a_{l,\ast})
}
\right].
\end{align*}
Also, by Taylor's formula with integral remainder, for \(\omega_\delta=\omega-\omega_\ast\),
\begin{align*}
a_{l,X}(\omega)
=
a_{l,\ast}+J_{l,\ast}\omega_\delta+
\int_0^1(1-s)\nabla_\omega^2a_{l,X}(\omega_\ast+s\omega_\delta)[\omega_\delta,\omega_\delta]\dd s.
\end{align*}
After integrating against \(\rho\), this gives
\begin{align*}
z_l^\rho(X)-a_{l,\ast}
=
J_{l,\ast}m
+
R_l,
\qquad
\norm{R_l}
\leq
\frac12\sup_{\omega\in B(\omega_\ast;r)}\norm{\nabla_\omega^2a_{l,X}(\omega)}
\int\norm{\omega_\delta}^2\rho(\dd\omega)
\leq
C M_2(\rho),
\end{align*}
where
\begin{align*}
R_l=
\int
\int_0^1(1-s)\nabla_\omega^2a_{l,X}(\omega_\ast+s\omega_\delta)[
\omega_\delta,
\omega_\delta]
\dd s\,\rho(\dd\omega).
\end{align*}
Therefore
\begin{align*}
\inner{
J_{l,\ast}m
}{
\softmax(z_l^\rho)-\softmax(a_{l,\ast})
}
=&
\inner{
z_l^\rho-a_{l,\ast}
}{
\softmax(z_l^\rho)-\softmax(a_{l,\ast})
}\\
&-
\inner{
R_l
}{
\softmax(z_l^\rho)-\softmax(a_{l,\ast})
}.
\end{align*}
The first term is nonnegative because softmax is the gradient of the convex log-sum-exp function and is therefore monotone. The second term is bounded below by
\begin{align*}
-C\norm{R_l}\norm{z_l^\rho-a_{l,\ast}}
\geq
-CrM_2(\rho),
\end{align*}
since \(\supp(\rho)\subset B(\omega_\ast;r)\). Hence
\begin{align}\label{Ineq:B}
B(\rho)
\geq
-CrM_2(\rho).
\end{align}
Choosing \(r>0\) smaller if necessary, \eqref{Ineq:A} and \eqref{Ineq:B} gives
\begin{equation}
\label{eq:local-coercivity}
\int
\inner{\omega-\omega_\ast}{
\nabla_\omega\Phi(\omega;\rho)
}
\rho(\dd\omega)
\geq
c_\ast M_2(\rho)
\end{equation}
for some \(c_\ast>0\).

We now use a bootstrap argument to show that the support remains in the neighborhood. From \eqref{eq:characteristic},
\begin{align*}
\frac{\dd}{\dd t}\frac{1}{2}\norm{W_t-\omega_\ast}^2
=
-2\zeta(t)
\inner{
W_t-\omega_\ast
}{
\nabla_\omega\Phi(W_t;\rho_t)
}.
\end{align*}
We now make the bootstrap argument explicit. For \(\rho\) supported in \(B(\omega_\ast;r)\) define
\begin{align*}
F(\omega,\rho)
=
\inner{\omega-\omega_\ast}{\nabla_\omega\Phi(\omega;\rho)}.
\end{align*}
Since \(\nabla_\omega\Phi(\omega_\ast;\rho_\ast)=0\), write
\begin{align*}
\nabla_\omega\Phi(\omega;\rho)
=&\
\bigl[\nabla_\omega\Phi(\omega;\rho)-\nabla_\omega\Phi(\omega_\ast;\rho)\bigr]
+
\bigl[\nabla_\omega\Phi(\omega_\ast;\rho)-\nabla_\omega\Phi(\omega_\ast;\rho_\ast)\bigr].
\end{align*}
For the first bracket, the local Hessian bound gives, after reducing \(r\) if necessary,
\begin{align*}
\inner{\omega-\omega_\ast}{
\nabla_\omega\Phi(\omega;\rho)-\nabla_\omega\Phi(\omega_\ast;\rho)
}
\geq
\frac{\lambda_0}{2}\norm{\omega-\omega_\ast}^2.
\end{align*}
For the second bracket, the Lipschitz dependence of \(\nabla_\omega\Phi(\omega_\ast;\rho)\) on the measure argument gives
\begin{align*}
\norm{
\nabla_\omega\Phi(\omega_\ast;\rho)-\nabla_\omega\Phi(\omega_\ast;\rho_\ast)
}
\leq
C W_1(\rho,\delta_{\omega_\ast}).
\end{align*}
Multiplying by \(\norm{\omega-\omega_\ast}\) and combining the two estimates yields, for \(\omega\in B(\omega_\ast;r)\),
\begin{align*}
F(\omega,\rho)
&\geq
\frac{\lambda_0}{2}\norm{\omega-\omega_\ast}^2
-C\norm{\omega-\omega_\ast}\,W_1(\rho,\delta_{\omega_\ast}).
\end{align*}
Since \(W_1(\rho,\delta_{\omega_\ast})\leq M_1(\rho):=\int\norm{\omega-\omega_\ast}\rho(\dd\omega)\), it follows that on the annulus
\begin{align*}
B(\omega_\ast;r)\setminus B(\omega_\ast;r/2)
\end{align*}
one has
\begin{align*}
F(\omega,\rho)
\geq
\frac{\lambda_0r^2}{8}-CrM_1(\rho).
\end{align*}
Choose \(r_0>0\) so small that \(r_0<r/2\) and \(Cr_0\leq\lambda_0 r/16\). If initially \(\supp(\rho_0)\subset B(\omega_\ast;r_0)\), then \(M_1(\rho_0)\leq r_0\) and \(M_2(\rho_0)\leq r_0^2\). On any time interval on which \(\supp(\rho_t)\subset B(\omega_\ast;r)\), the moment estimate \eqref{Ineq:M_2} below implies \(M_2(\rho_t)\leq M_2(\rho_0)\), hence \(M_1(\rho_t)\leq M_2(\rho_t)^{1/2}\leq r_0\). Therefore \(F(\omega,\rho_t)>0\) on the annulus \(B(\omega_\ast;r)\setminus B(\omega_\ast;r/2)\). Along any characteristic \(W_t\),
\begin{align*}
\frac{\dd}{\dd t}\frac12\norm{W_t-\omega_\ast}^2
=
-2\zeta(t)F(W_t,\rho_t),
\end{align*}
so a characteristic cannot cross outward through \(\partial B(\omega_\ast;r)\): at such a crossing it lies in the annulus and the radial derivative is strictly negative. By the standard continuation argument applied up to the first exit time from \(B(\omega_\ast;r)\), no exit can occur. This proves that the support remains inside \(B(\omega_\ast;r)\), closing the bootstrap.

Finally, applying the weak formulation \eqref{eq:weak-pde} to \(f(\omega)=\norm{\omega-\omega_\ast}^2\) and using \eqref{eq:local-coercivity}, we obtain
\begin{align*}
\frac{\dd}{\dd t}M_2(\rho_t)
&=
-4\zeta(t)
\int
\inner{\omega-\omega_\ast}{
\nabla_\omega\Phi(\omega;\rho_t)
}
\rho_t(\dd\omega)\\
&\leq
-4c_\ast\zeta(t)M_2(\rho_t).
\end{align*}
Gronwall's inequality gives
\begin{align}\label{Ineq:M_2}
M_2(\rho_t)
\leq
\exp\left(-4c_\ast\int_0^t\zeta(s)\dd s\right)M_2(\rho_0),
\end{align}
which is exactly \eqref{eq:dirac-exp-rate}.
\end{proof}

\begin{remark}
The condition \(H_0\succeq\lambda_0I\) is a convenient sufficient condition, not a necessary one. The reason is that a pure translation of a Dirac mass changes both the evaluation point and the measure argument. If \(\rho_h=\delta_{\omega_\ast+h}\), then Taylor expansion gives
\begin{align*}
\nabla_\omega\Phi(\omega_\ast+h;\rho_h)
=
H_{\mathrm{tr}}h+o(\norm{h}),
\end{align*}
where
\begin{align*}
H_{\mathrm{tr}}
=
H_0+
\E_X\left[
\frac1L\sum_{l=1}^L
J_{l,\ast}^\top S_{l,\ast}J_{l,\ast}
\right].
\end{align*}
Since \(S_{l,\ast}=\diag(p_{l,\ast})-p_{l,\ast}p_{l,\ast}^\top\succeq0\), the second term is positive semidefinite. Thus a direction that is weakly unstable for the Hessian \(H_0\) alone may still be stabilized along translations by the cross-entropy curvature in logit space. The theorem uses the stronger condition \(H_0\succ0\) because it also controls non-translational spreading of the measure around \(\omega_\ast\).
\end{remark}

\subsection{A translation instability criterion}

The next result gives a complementary instability condition.
Let
\begin{align*}
p_{l,\ast}(X)=\softmax(\psi_l(X;\omega_\ast)),
\qquad
S_{l,\ast}(X)=\diag(p_{l,\ast}(X))-p_{l,\ast}(X)p_{l,\ast}(X)^\top.
\end{align*}
Define the translation Hessian
\begin{equation}
\label{eq:translation-hessian}
H_{\mathrm{tr}}
=
H_0
+
\E_X
\left[
\frac{1}{L}
\sum_{l=1}^L
\nabla_\omega\psi_l(X;\omega_\ast)^\top
S_{l,\ast}(X)
\nabla_\omega\psi_l(X;\omega_\ast)
\right].
\end{equation}

\begin{theorem}[Translation instability]
\label{thm:instability}
Assume \(\rho_\ast=\delta_{\omega_\ast}\) is stationary. If there exists \(v\in\R^{d_\omega}\) with \(\norm{v}=1\) such that
\begin{align*}
v^\top H_{\mathrm{tr}}v<0,
\end{align*}
then \(\rho_\ast\) is Lyapunov unstable in \(W_2\). More precisely, for all sufficiently small \(\varepsilon>0\), the solution starting from
\begin{align*}
\rho_0^\varepsilon=\delta_{\omega_\ast+\varepsilon v}
\end{align*}
initially moves away from \(\rho_\ast\):
\begin{align*}
\left.
\frac{\dd}{\dd t}
W_2(\rho_t^\varepsilon,\rho_\ast)^2
\right|_{t=0}
>0.
\end{align*}
\end{theorem}

\begin{remark}[Comparison with saddle instability]
This instability criterion is the mean-field analogue of the elementary fact that a finite-dimensional gradient flow moves away from an equilibrium along a negative Hessian direction. The matrix \(H_{\mathrm{tr}}\), rather than only \(H_0\), appears because translating a Dirac mass changes both the location \(\omega\) and the law \(\delta_\omega\). Hence the criterion detects instability of the coupled mean-field vector field, not merely instability of \(\Phi(\cdot;\rho_\ast)\) with \(\rho_\ast\) frozen.
\end{remark}

\begin{proof}
Since the initial condition is a Dirac mass, the solution remains a Dirac mass for short time:
\begin{align*}
\rho_t^\varepsilon=\delta_{\omega^\varepsilon(t)}.
\end{align*}
The characteristic equation is
\begin{align*}
\dot\omega^\varepsilon(t)
=
-2\zeta(t)
\nabla_\omega\Phi(\omega^\varepsilon(t);\delta_{\omega^\varepsilon(t)}).
\end{align*}
Let
\begin{align*}
r_\varepsilon(t)=\norm{\omega^\varepsilon(t)-\omega_\ast}^2.
\end{align*}
Using stationarity at \(\omega_\ast\) and Taylor expansion,
\begin{align*}
\nabla_\omega\Phi(\omega_\ast+\varepsilon v;\delta_{\omega_\ast+\varepsilon v})
=
\varepsilon H_{\mathrm{tr}}v+o(\varepsilon).
\end{align*}
Therefore
\begin{align*}
r_\varepsilon'(0)
=
-4\zeta(0)
\inner{
\varepsilon v
}{
\varepsilon H_{\mathrm{tr}}v+o(\varepsilon)
}
=
-4\zeta(0)\varepsilon^2 v^\top H_{\mathrm{tr}}v
+
o(\varepsilon^2).
\end{align*}
If \(v^\top H_{\mathrm{tr}}v<0\), then \(r_\varepsilon'(0)>0\) for all sufficiently small \(\varepsilon\), proving instability.
\end{proof}

\section{Discussion of convergence to steady states}

The results above clarify the answer to the convergence question for the limiting PDE \eqref{eq:pde}.

\paragraph{General case.}
Without additional assumptions, the PDE is a gradient-flow-type dynamics and the risk is decreasing. However, this alone does not imply convergence to a unique stationary solution. What can be proved generally is convergence to the stationary set, in the sense that all omega-limit points are stationary.

\paragraph{Convergence to a single steady state.}
Convergence to a single stationary solution follows if the omega-limit set cannot contain a nontrivial continuum of stationary measures. This is guaranteed, for example, if the relevant stationary set is finite or totally disconnected.

\paragraph{Rates.}
Explicit convergence rates require quantitative inequalities relating the dissipation \(\mathcal I(\rho)\) to the energy gap. A Kurdyka--{\L}ojasiewicz inequality gives exponential or algebraic rates depending on the exponent \(\theta\), where the exponent \(\theta\) measures the flatness of the energy landscape near the limiting stationary set.

\paragraph{Stability.}
A stationary solution is locally exponentially stable if the vector field is strongly monotone in \(W_2\) near it. For Dirac stationary measures, a verifiable sufficient condition is positive definiteness of the local Hessian \(H_0=\nabla_\omega^2\Phi(\omega_\ast;\delta_{\omega_\ast})\). Conversely, if the translation Hessian \(H_{\mathrm{tr}}\) has a negative direction, then the Dirac stationary solution is unstable.

These statements are important because the set of stationary points is generally larger than the set of global minimizers. Therefore, mean-field convergence to stationarity should not be confused with convergence to a globally optimal predictor unless additional landscape assumptions are imposed. 

One sufficient global condition is a Polyak-{\L}ojasiewicz (PL) inequality relative to the global optimum
\begin{align*}
\mathcal I(\rho)
\geq
2\lambda\bigl(\calE(\rho)-\calE_{\min}\bigr),
\qquad
\calE_{\min}=\inf_{\nu\in\calP(K)}\calE(\nu),
\end{align*}
on the invariant region explored by the flow. In that case,
\begin{align*}
\calE(\rho_t)-\calE_{\min}
\leq
\bigl(\calE(\rho_0)-\calE_{\min}\bigr)
\exp\left(-4\lambda\int_0^t\zeta(s)\dd s\right),
\end{align*}
and every limit point is a global minimizer. If, in addition, minimizers are identifiable at the level of logits, for example if there exists a target logit field \(z^\ast\) and a constant \(c>0\) such that near the minimizer set
\begin{align*}
\calE(\rho)-\calE_{\min}
\geq
c\,\E_X\left[\frac1L\sum_{l=1}^L
\norm{z_l^\rho(X)-z_l^\ast(X)}^2
\right],
\end{align*}
then convergence of the energy also implies convergence of the predictor logits. Without such PL, identifiability, or no-spurious-stationary-point assumptions, the results of this paper guarantee stationarity but not global optimality.

Alternatively, convergence to the stationary set is already convergence to the set of global mean-field minimizers if every stationary measure in the omega-limit set satisfies the variational support condition of \Cref{thm:minimizer-support}. We spell out what this means in the present attention model. For a stationary measure \(\bar\rho\), set
\begin{align*}
r_l^{\bar\rho}(X)=\softmax(z_l^{\bar\rho}(X))-Y_l(X).
\end{align*}
Using
\begin{align*}
\psi_l(X;\omega)
=
\bigl[A_\omega(X)XW_VW_OW_{\mathrm{out}}\bigr]_{l,:}
\in\R^{d_v},
\end{align*}
The support condition \eqref{eq:minimizer-support} says that every active head parameter minimizes the scalar first variation \(\Phi(\cdot;\bar\rho)\) over \(K\).  Equivalently, for every \(\omega'\in K\) and every \(\omega\in\supp(\bar\rho)\),
\begin{align*}
\Phi(\omega';\bar\rho)-\Phi(\omega;\bar\rho)\geq 0.
\end{align*}
Using the displayed expression for \(r_l^{\bar\rho}\) and the definition of \(\Phi\), this difference is
\begin{align*}
\Phi(\omega';\bar\rho)-\Phi(\omega;\bar\rho)
=
\E_X\left[
\frac1L\sum_{l=1}^L
\inner{r_l^{\bar\rho}(X)}{
\psi_l(X;\omega')-\psi_l(X;\omega)
}
\right].
\end{align*}
Therefore \eqref{eq:minimizer-support} implies the global variational inequality
\begin{align}
\label{eq:stationary-global-optimality-vi}
\E_X\left[
\frac1L\sum_{l=1}^L
\inner{r_l^{\bar\rho}(X)}{
\psi_l(X;\omega')-\psi_l(X;\omega)
}
\right]
\geq 0
\end{align}
for every \(\omega'\in K\) and every \(\omega\in\supp(\bar\rho)\), hence in particular for \(\bar\rho\)-almost every active head \(\omega\). Conversely, if \eqref{eq:stationary-global-optimality-vi} holds for every \(\omega'\in K\) and for \(\bar\rho\)-almost every \(\omega\), then \(\Phi(\omega';\bar\rho)\geq\Phi(\omega;\bar\rho)\) for all \(\omega'\in K\) and \(\bar\rho\)-almost every active \(\omega\). Since \(\Phi(\cdot;\bar\rho)\) is continuous on the compact set \(K\), this inequality extends from a full \(\bar\rho\)-measure set to the whole support \(\supp(\bar\rho)\). Thus every point in \(\supp(\bar\rho)\) belongs to \(\argmin_K\Phi(\cdot;\bar\rho)\), which is exactly \eqref{eq:minimizer-support}. Thus a stationary measure is globally minimizing precisely when no alternative attention head, with its own query, key, value, and output matrices, can reduce the residual-weighted vocabulary-logit score. Stationarity alone only gives the infinitesimal version of this condition at active heads, namely \(\nabla_\omega\Phi(\omega;\bar\rho)=0\), and therefore does not rule out another head parameter \(\omega'\) with smaller value of \(\Phi(\omega';\bar\rho)\).

Several useful sufficient conditions follow from this formulation. First, if for every \(\bar\rho\in\Omega(\rho_0)\) the scalar function \(\omega\mapsto\Phi(\omega;\bar\rho)\) has no spurious critical points on the invariant compact set, meaning that every point satisfying \(\nabla_\omega\Phi(\omega;\bar\rho)=0\) is a global minimizer of \(\Phi(\cdot;\bar\rho)\), then all stationary measures in \(\Omega(\rho_0)\) satisfy the support condition. This condition is implied, for instance, by convexity or geodesic convexity of \(\Phi(\cdot;\bar\rho)\) on the relevant parameter region together with the appropriate boundary normal-cone condition. Second, the same conclusion holds if the attainable head-feature set
\begin{align*}
\mathcal H_X=
\left\{
\bigl(\psi_1(X;\omega),\ldots,\psi_L(X;\omega)\bigr):\omega\in K
\right\}
\end{align*}
is effectively convex after the parameterization \(\omega\mapsto A_\omega XW_VW_OW_{\mathrm{out}}\), and the active heads minimize the linear functional induced by \((r_l^{\bar\rho}(X))_{l=1}^L\) over this feature set. Third, the support condition is automatic in the degenerate zero-dual-residual case in which
\begin{align*}
\E_X\left[
\frac1L\sum_{l=1}^L
\inner{r_l^{\bar\rho}(X)}{
\psi_l(X;\omega')-\psi_l(X;\omega)
}
\right]=0
\end{align*}
for all \(\omega',\omega\in K\); equivalently, \(\Phi(\cdot;\bar\rho)\) is constant on \(K\). This includes the formal case in which the residual field is orthogonal to all attainable variations of the attention-head logit feature.

Conversely, \eqref{eq:stationary-global-optimality-vi} gives a necessary and sufficient condition. In particular, if \(\bar\rho\) is stationary but there exist an active head \(\omega\in\supp(\bar\rho)\) and an alternative head \(\omega'\in K\) such that
\begin{align*}
\E_X\left[
\frac1L\sum_{l=1}^L
\inner{r_l^{\bar\rho}(X)}{
\psi_l(X;\omega')-\psi_l(X;\omega)
}
\right]<0,
\end{align*}
then \(\bar\rho\) fails the variational support condition and is not a global minimizer. Hence a necessary condition for all stationary measures in \(\Omega(\rho_0)\) to be global minimizers is the absence, at every such stationary measure, of residual-decreasing alternative heads in the feature class generated by the formula for \(\psi_l\). It is also necessary that \(\Phi(\omega;\bar\rho)\) be constant at its minimum on \(\supp(\bar\rho)\): if two active heads have different values of \(\Phi(\cdot;\bar\rho)\), or if their common value is strictly larger than \(\min_{\omega'\in K}\Phi(\omega';\bar\rho)\), then the support condition fails.

\end{document}